\documentclass[11pt]{amsart}
\usepackage{mathptmx}
\usepackage{amsmath}
\usepackage{amscd}
\usepackage{amssymb}
\usepackage{amsthm}
\usepackage{xspace}
\usepackage[all,tips]{xy}
\usepackage[dvips]{graphicx}
\usepackage{verbatim}
\usepackage{syntonly}
\usepackage{hyperref}
\usepackage{arydshln}
\usepackage{lineno,color}

\theoremstyle{plain}
\newtheorem{thm}{\underline{Theorem}}[section]

\newtheorem{lemma}[thm]{\underline{Lemma}}

\newtheorem*{defn*}{\underline{Definition}}

\newtheorem*{cor*}{\underline{Corollary}}
\newtheorem*{note*}{\underline{Note}}

\newcommand{\eps}{\varepsilon}

\newcommand{\al}{\alpha}

\newcommand{\ga}{\gamma}

\newcommand{\la}{\lambda}

\newcommand{\innp}[1]{\left< #1 \right>}

\newcommand{\pr}[1]{\left( #1 \right) }

\newcommand{\Fr}[1]{\ensuremath{\mathfrak{#1}}}

\newtheorem{lemma_bis}{\underline{Lemma bis}}[thm]

\newcommand{\nsub}{\trianglelefteq}
\usepackage{fancyhdr}

\begin{document}

\title{\textbf{"The Edmonton Notes on Nilpotent Groups" by Philip Hall}}
\author{Mark Pengitore}
\maketitle

\begin{abstract}

This note is a reproduction of the well known and historically significant series of notes called "The Edmonton Notes on Nilpotent Groups" based on lectures given by Philip Hall using the copy found in the Queen Mary College Mathematics Notes series. We use the original numbering of statements, definitions, footnotes, proofs, abstract, etc. to recover as much of the original document as possible.
\end{abstract}
\tableofcontents


\begin{center}
\textbf{Foreword}
\end{center}
These notes are based on a series of lectures given by Professor Philip Hall at the Canadian Mathematical Congress Summer Seminar at the University of Albert from August 12 to August 30 in 1957. They were published by the Canadian Mathematical Society initially and were circulated in small numbers. These notes were then later published in Queen Mary College Mathematics Notes in 1969 \cite{hall_queen_mary}. 
These typewritten notes have largely disappeared from circulation. Fortunately, a version of these notes were included in the 1988 volume of collected works of Philip Hall \cite{hartley}. However since then, this book has become hard to find outside some libraries.

I was fortunate to receive a digital scan of the Edmonton Notes on Nilpotent Groups by Philip Hall which appeared in the Queen Mary College Mathematics Notes published 1969 through email correspondence with Khalid Bou-Rabee while I was in graduate school at Purdue University. Recently I took the initiative to typeset these historically significant notes because I believe it would be of great value for the group theory community and the broader math community to have access to these valuable notes. Since I am not the original author of the notes nor a member of the estate of Philip Hall, I began the process of obtaining permission to post my typeset version of these notes to the arXiv. Through contacting various mathematical organizations such as the American Mathematical Society, the Canadian Mathematical Society, and the Queen Mary University of London mathematics department, I found out that Philip Hall's literary estate was inherited by Dr. Jim Roseblade who is currently an Emeritus Fellow of Jesus College of the University of Cambridge. In 2019, he gave Hall's literary collection to the Cambridge University Library including any copyrights to anything in the collection. Representatives from Cambridge University library told me that they have 38 boxes of work and notes from Philip Hall which is uncatalogued. My original plan was to visit the University of Cambridge to search these boxes for any clues on these notes. Fortunately, representatives from the Queen Mary University of London school of mathematics have told me that they consider the Queen Mary Mathematics Notes to be defunct and have given me permission to post these notes.

As mentioned above, these notes are historically significant and are very influential in the theory of nilpotent groups with over 250 citations. As reviews indicate, much but not all of the material in these lecture notes have appeared in various subsequent research papers and books like Gilbert Baumslag's book "Notes on nilpotent Groups \cite{baumslag}" and "Polycyclic Groups" by Dan Segal \cite{segal} for instance. Therefore, I believe that it is a great value to have these notes available to be read and referenced by whomever may need them. I finish by stressing that I claim no credit for the material found in this note and that it is all originally due to Philip Hall.

In this typesetting of these notes, I strove to stay as true to the original document as much as possible including the original notation and language that were used. I also included the foreword as it originally appeared in the notes. The differences between this version and the original notes are as follows:
\begin{itemize}
    \item The appendix has been changed into footnotes which are referenced and appear close to where they are originally appear as in the original document to ensure readability.
    \item I adjusted spacing throughout the article to improve readability.
    \item In the original notes, curly brackets $\left\{ A \right\}$ indicate a group generated by the set $A$. To improve readability and for conveniences of the modern audience, we use $\left<A\right>$ to indicate the group generated by the set $A$.
    \item I added a notation subsection for readers to reference while going through these notes.
\end{itemize}\

{\bf Acknowledgments} I want to thank a number of people. I thank Khalid Bou-Rabee for providing me with a digital scan of these invaluable notes which allowed me to start the process of giving access to these notes to the broader mathematical community. I thank Jean-Francois Lafont for giving me advice on what mathematical organizations to contact for information on who controlled the intellectual rights of these notes.  I thank Maura Goss who is a rights, licensing, and permissions specialist of the American Mathematical Society for their extended interactions with me via email. Their emails were valuable in ascertaining who to contact in order to obtain permission to post these notes. I thank Frank Bowles of the department of archives and modern manuscripts of Cambridge University Library who directed me to contact Jim Roseblade. I thank Jim Roseblade for explaining the publication history and intellectual rights history of these notes to me. I thank Matt Fayers who is the director of education operations of the school of mathematics of Queen Mary University of London who gave me permission to post these notes. Finally, I thank Peter Kropholler for allowing me to fill in the portions of these notes which were unreadable in my electronic copy with his paper copy of the Edmonton Notes on Nilpotent Groups. \\

\begin{center}\textbf{Notation}
\end{center}
Let $G$ be a group, and let $A, B \leq G$ be nonempty subsets. Let $x,y \in G$ be elements.
\begin{itemize}
    \item $\left<A\right>$ is the subgroup group generated by the set $A$.
    \item $|G|$ is the cardinality of the group $G$.
    \item $[x,y] = x^{-1}y^{-1}xy$ and $[x_1, x_2, \ldots, x_k] = [[x_1, \ldots, x_{k-1}], x_k]$. We define $[A,B]$ to be the subgroup generated by $[a,b]$ where $a \in A$ and $b \in B$. We similarly define $[A_1, \ldots, A_k]$.
    \item $C_G(A)$ is the centralizer of the set $A$ in $G$.
    \item $x^y =y^{-1}xy$ and $A^B$ is the set of all elements $a^b$ for $a \in A$ and $b \in B$.
    \item $G'$ is the commmutator subgroup of $G$.
    \item $G^{(i)}$ is the $i$-th term of the derived series of $G$.
    \item $\gamma_i(G)$ is the $i$-th term of the lower central series, and when $G$ is clear from context, we write $\Gamma_i = \gamma_i(G)$.
    \item $\zeta_i(G)$ is the $i$-th step of the upper central series, and when $G$ is clear from context, we write $Z_i = \zeta_i(G).$\\
\end{itemize}

\begin{center}
\title{\textbf{The Edmonton Notes on Nilpotent Groups}} \\
$\:$ \newline
\author{Philip Hall}\\
\end{center}

These notes reproduce a series of lectures given by Professor P. H. of the Summer Institute of the Canadian Mathematical Congress held at Edmonton in 1957. They were originally published by the Canadian Mathematical Society but have been unobtainable for many years. The present edition differs from the first only in having a rewritten $\S 4.$  We have also added references to some relevant papers that have appear since 1957\\

   We are extremely grateful to the Canadian Mathematical Society and, of course, to Professor Hall for their permission to reissue these notes.\\
   
  \begin{flushright} August 1969\\
   
   Mathematics Department,\\
   Queen Mary College

   \end{flushright}

\section{Central Series}
\begin{defn*}
 Let $H \nsub G$, $K \nsub G$, and $K \leq H$. If $H / K$ is contained in the centre of $G/K$, then $H/ K$ is called a \underline{central factor of $G$}. A group $G$ is called \underline{nilpotent} if and only if it has a finite series of normal subgroups: 
\begin{equation}\label{first_central_series} 
G = G_0 \geq G_1 \geq \ldots \geq G_r = 1\tag{1}
\end{equation} 
such that $G_{i-1} / G_i$ is a central factor of $G$ for each $i = 1, 2, \ldots, r$. The number $r$ is the formal \underline{length} of the series (\ref{first_central_series}). The smallest value of $r$ for any such central series of $G$ is called the \underline{class} of $G$.
\end{defn*}

Thus Abelian groups are the same as nilpotent groups of class $1$, except that unit groups are of class $0$.

\begin{defn*} If $x$ and $y$ are elements of a group $G$, their \underline{commutator} $x^{-1}  y^{-1}  x  y$ is written $[x,y]$. If $X,Y$ are subgroups of $G$, then $[X,Y]$ is the subgroup generated by all the commutators $[x,y]$ with $x \in X$, $y \in Y$.
\end{defn*}

Since $[y,x] = [x,y]^{-1}$, we have $[Y,X] = [X,Y]$. By convention, for $n > 2$, 
$$
[x_1, x_2, \ldots, x_n] = [[x_1, \ldots,x_{n-1}], x_n],
$$ 
and similarly for subgroups.

To say that $H / K$ is a central factor of $G$ is to say that 
$$
x y = y  x\text{ mod } K
$$ for all $x \in H$, $y \in G$, or again, that 
$$
[H,G] \leq K.
$$

Thus (\ref{first_central_series}) is a central series of $G \Leftrightarrow [G_{i-1}, G] \leq G_i$ for $i = 1, 2, \ldots, r$.

\begin{defn*} The \underline{upper central series}
$$
1 =\zeta_0(G) \leq \zeta_1(G) \leq \ldots 
$$ 
of any group $G$ is defined inductively by the rule that $\zeta_{i+1}(G)/ \zeta_{i}(G)$ is the centre of $G / \zeta_i(G)$. Thus $\zeta_1(G)$ is the centre of $G$.

The \underline{lower central series} written 
$$
G= \gamma_1(G) \geq \gamma_2(G) \geq \ldots
$$ 
of $G$ is defined by the rule that $\gamma_{i+1}(G) = [\gamma_i(G), G]$.
\end{defn*}

When $G$ is understood from the context, we write $Z_i = \zeta_i(G)$ and $\Gamma_j = \gamma_j(G)$. Thus $\Gamma_j = [G, \ldots, G]$ with $j$ terms. In particular, $\Gamma_2 = [G, G]$ is the \underline{first derived group} of $G$, often written $G'$.

If (\ref{first_central_series}) is any central series of $G$, then by induction on $i,j$
$$
Z_i \geq G_{r-i}, \quad \Gamma_{j+1} \leq G_j \quad (i,j = 0, 1, 2, \ldots, r).
$$
Hence,
\begin{lemma}\label{lemma1.1}
$G$ is nilpotent of class $c$ if and only if
$$
 Z_0 < Z_1 < \ldots < Z_c = G \Leftrightarrow G = \Gamma_1 > \Gamma_2 > \ldots > \Gamma_{c+1} = 1.
$$
\end{lemma}

For any group $G$, $\zeta_i(G)$ is a characteristic subgroup. Since commutators map into commutators in any homomorphic mapping of a group, $\gamma_j(G)$ is not only characteristic but fully invariant in $G$ (it maps into itself under any endomorphism of $G$). More generally

\begin{lemma}\label{lemma1.2}
$K \nsub G \Rightarrow \gamma_j(G/K) =  K \gamma_j(G)/ K$.\\ In particular, $G / K$ is nilpotent $\Leftrightarrow $ $K$ contains some term of the lower central series of $G$.
\end{lemma}
\begin{lemma}\label{lemma1.3}
If $G$ is nilpotent of class $c$, then every subgroup and every factor group of $G$ is nilpotent of class at most $c$. If $c > 1$, and $x \in G$, then $H = \left<x,G'\right>$ is of class at most $c-1$.
\end{lemma}
\begin{proof}\noindent For $$
G' \leq \zeta_{c-1}(G) \cap H \leq \zeta_{c-1}(H) \cap H \leq \zeta_{c-1}(H)
$$
so $H /\zeta_{c-1}(H)$ is cyclic, hence $1$.
\end{proof}

A direct product of nilpotent groups $G_\alpha$ is nilpotent if and only if there is a bound for the classes of the direct factors.
\begin{lemma}\label{lemma1.4}
If $|G| = p^n > 1$, where $p$ is a prime, then $Z_1 > 1.$ Hence $G$ is nilpotent of class $<n.$
\end{lemma}

\begin{defn*} A \underline{word} $\theta(x)$ in the variables $x_1, x_2, \ldots$ is a formal expression 
$$
\theta(x) =  x_{i_1}^{r_1}, x_{i_2}^{r_2}, \ldots x_{i_n}^{r_n}
$$ 
with integers $r_1,r_2,\ldots r_n$. If $G$ is any group, $\theta(G)$ is the subgroup generated by all elements of $G$ expressible in the form $\theta(a)$ for suitable choices of $a_1, a_2, \ldots$ in $G$. $\theta(G)$ is fully invariant in $G$.
\end{defn*}

\begin{lemma}\label{lemma1.5}
Let $\theta(x), \phi(y)$ be any two words involving the variables $x_1, \ldots, x_m$ and $y_1,\ldots, y_n$ at most. Define 
$$
\psi(x) = [\theta(x_1, \ldots, x_m), \phi(x_{m+1}, \ldots, x_{m+n})].
$$ Then for any group $G$, $$\psi(G) = [\theta(G), \phi(G)].$$
\end{lemma}
\begin{proof}
Clearly $\psi(G) \leq [\theta(G), \phi(G)]$. But $\text{mod } \phi(G)$, every $\theta$-value in $G$ commutes with every $\psi$-value.
\end{proof}

\begin{cor*}The terms $\gamma_n(G)$ of the lower central series of $G$ are \underline{verbal subgroups} in this sense, if we define 
$$
\gamma_n(x) = [x_1, x_2, \ldots, x_n].
$$
\end{cor*}

\begin{note*}
Let 
$$
\prod_{\lambda \in \Lambda} G_\lambda = C
$$ 
be a Cartesian product, $\Phi_\lambda = \phi(G_{\lambda})$, $\phi$ any word. In general 
$$
\phi(C) < \prod_{\lambda \in \Lambda} \Phi_\lambda.
$$
\end{note*}

\begin{defn*} $t^y = y^{-1}t y$. $A^B$ is the set of all elements $a^b$ for $a$ in $A$, $b$ in $B$.\end{defn*}

\begin{lemma}\label{lemma1.6}
Let $G = \left<A\right>$ and let $H = \left<B^G\right> \nsub G$. Then $[H,G] = \left<C^G\right>$ where $C$ is the set of all $[a,b]$ with $a$ in $A$, $b$ in $B$.
\end{lemma}
\begin{proof}
\noindent Write $K = \left<C^G\right>$. Clearly $K \nsub G$ and $K' \leq [H,G]$ since $H \nsub G$. But $\text{mod } K$, every generator $a$ of $G$ commutes with every $b \in B$. Hence each $bK$ belongs to the centre of $G/K$. Thus $H / K$ is contained of $G/K$; $[H, G] \leq K$.
\end{proof}

\begin{lemma}\label{lemma1.7}
Let $G = \left<A\right>$. Then $\Gamma_n = \gamma_n(G)$ is generated by $\Gamma_{n+1}$ together with all of the commutators $[a_1, \ldots, a_n]$ with each $a_i$ in $A$.
\end{lemma}
\begin{proof}
For by (\ref{lemma1.6}), $\Gamma_n$ is generated by these commutators $[a_1,\ldots,a _n]$  together with all of their conjugates in $G$. But $u^x = u  [u, x]$ and $[u,x] \in \Gamma_{n+1}$ for all $x \in G$.
\end{proof}
\begin{thm}\label{theorem1.8}
Let $G = \left<a_1,a_2, \ldots, a_q\right>$ be a finitely generated nilpotent group. Then $G$ has a central series $$G = G_0 \geq G_1 \geq \ldots \geq G_r = 1$$ with each $G_{i-1} / G_i$ cyclic.
\end{thm}
\begin{proof}
\noindent For let $G$ have class $c$, so that $\Gamma_{c+1} = 1$. For each $k=1,2, \ldots, c$ there are finite number of commutators of the form $$[a_{i_1}, a_{i_2}, \ldots, a_{i_k}],$$ say $u_{k_1}, u_{k_2}, \ldots, u_{k_m}$, where $m = m(k)$. Let 
$$
G_{k_j} = \left<\Gamma_{k+1}, u_{k_1}, u_{k_2}, \ldots, u_{k_j}\right>.
$$ If for each $k$ we interpolate the groups $G_{k_1}, G_{k_2}, \ldots, G_{k_{m-1}}$ between $\Gamma_{k+1} = G_{k_0}$ and $\Gamma_k = G_{k_m}$ we obtain a central series of $G$ with cyclic factor groups as required.
\end{proof}

\begin{defn*} A group $G$ is called \underline{polycyclic} if it has a series of subgroups $$
G = G_0 \geq G_1 \geq \ldots \geq G_r = 1
$$
such that, for each $i = 1, 2, \ldots, r$, $G_i \nsub G_{i-1}$ and $G_{i-1} / G_i$ is cyclic. If $G$ has a series of this kind for which each $G_i \nsub G$, then $G$ is called \underline{supersoluble}.

In a polycyclic group the number of cyclic factors $G_{i-1} / G_i$ $(i =1, \ldots, r)$ which are of infinite order will be called the \underline{rank} of $G$. It is independent of the particular series chosen.\footnote{Now frequently called the \underline{Hirsch number}}

A \underline{soluble} group $G$ is one which has a series of subgroups $G = G_0 \geq G_1 \geq \ldots \geq G_r = 1$ for which each $G_i \nsub G_{i-1}$ and each $G_{i-1} / G_i$ is abelian.

A group $G$ is said to satisfy the \underline{maximal condition} \textbf{Max} if all of its subgroups are finitely generated.
\end{defn*}

\begin{lemma}\label{lemma1.9}
Let $G$ be a finitely generated group. 
\begin{itemize}
\item[(i)] Finitely generated $+$ nilpotent $\Rightarrow$ supersoluble $\Rightarrow$ polycyclic $\Leftrightarrow$ soluble $+$ \textbf{Max}
\item[(ii)] Every supersoluble group $G$ has a nilpotent subgroup $K$ such that $[G:K]$ is finite and $K \geq G'.$
\end{itemize}
\end{lemma}
\begin{proof} (i) The first two implications are clear from (\ref{theorem1.8}) and the definitions. Suppose $G$ is a soluble group satisfying the maximal condition. Then in the series $$G = G_0 \geq \ldots \geq G_r = 1$$ with Abelian factor groups $G_{i-1} / G_i,$ each $G_{i-1}$ is finitely generated, so that we can interpolate between each $G_{i-1}$ and $G_i$ a finite number of subgroups, to produce a series with cyclic factor groups. Hence $G$ is polycyclic.

Conversely, if each $G_{i-1} / G_{i}$ is cyclic, let $H$ any subgroup of $G$ and define $H_i = H \cap G_i$ $(i=0,1, \ldots, r)$. Then each $$H_{i-1} / H_i \cong G_i  H_{i-1} /  G_i$$ is again cyclic and we can choose $x_i \in H_{i-1}$ so that $H_{i-1} = \left<H_i,x_i\right>$. Therefore, $H = \left<x_1,x_2, \ldots, x_r\right>$ is finitely generated. Therefore $G$ satisfies the maximal condition.

(ii) If $G$ is supersoluble, we may suppose that each $G_i \nsub G$ and each $G_{i-1} / G_i$ cyclic. Let $C_i$ be the centraliser of $G_{i-1} / G_i$ in $G$. By definition, $C_i$ is the largest subgroup of $G$ such that $[G_{i-1}, C_i] \leq G_i.$ Since $G_{i-1}$ and $G_i$ are normal in $G$, so is $C_i$ and $G/C_i$ is isomorphic with a subgroup of the group of automorphisms $A_i$ of the cyclic group $G_{i-1} / G_i$. Therefore $A_i$ and hence $G/C_i$ are finite Abelian groups. Let $$K = \bigcap_{i=1}^r C_i.$$ Then $[G:K]$ is finite and $G/K$ is Abelian, so $K \geq G'$. If $K_i = K \cap G_i$ we have 
$$
[K_i, K] \leq K \cap [G_i, C_i] \leq K \cap G_{i+1} = K_{i+1}.
$$ Thus $K = K_0 \geq K_1 \geq \ldots \geq K_r$ is a central series of $K$. So $K$ is nilpotent.
\end{proof}

We conclude this section by giving an illustration of the use of Lemma \ref{lemma1.1} in induction proofs.
\begin{thm}\label{theorem1.10}
If $G = \left<a_1, a_2, \ldots, a_r\right>$ is nilpotent, where $a_i$ is of finite order $m_i$ $(i=1,2, \ldots, r)$, then $G$ is finite and its order divides some power of $m = m_1 m_2 \ldots m_r$.
\end{thm}
\begin{proof}
Let $ A_i = \left<a_i^G\right>$ be the normal closure of $a_i$ in $G$. It is enough to prove that $|A_i|$ divides some power of $m_i$, for then $G$ is finite and $|G| =|A_1 A_2 \ldots A_r|$ divides $\prod_{i=1}^r |A_i|$. If $G$ is Abelian, then $|A_i| = m_i$ as required. Let $G$ be of class $c > 1$. By Lemma \ref{lemma1.3}, $A_i$ (which is contained in $\left<G', a_i\right>$) is of class less than $c$. Also $A_i$ is finitely generated by (\ref{lemma1.9}). Thus $A_i$ is generated by finitely many conjugates of $a_i$ (all of which have order $m_i$), and so we may assume (by induction on $c$) that $|A_i|$ is finite and divides some power of $m_i.$ The theorem now follows immediately, as noted above.
\end{proof}
\begin{cor*}
In a nilpotent group, periodic elements of co-prime order commute.
\end{cor*}
\begin{proof}
For if $x,y$ are such elements, apply (\ref{theorem1.10}) to $\innp{x,x^y}$ and to $\innp{y, y^x}$.
\end{proof}

\section{Theorems of Burnside, Wielandt, McLain, Hirsch, and Plokin}
In any group, $x^y = x[x,y]$ and so $x^{yz} = x[x,z][x,y]^z$ but also $x[x,yz]$. Similarly $(xy)^z = xy[xy,z]$ but also $x^zy^z = x[x,z]^y[y,z] = xy[x,z]^y[y,z]$. Thus
\begin{lemma}\label{lemma2.1}
$$
[x,yz] = [x,z][x,y]^z 
$$
$$
[xy,z] = [x,z]^y[y,z].
$$
\end{lemma}
\begin{note*} $$[y,z] = [x,y^{-1}]^y = [x^{-1},y]^x = [y^{-1}, x^{-1}]^{xy} = [x,y]^{-1}.$$ \end{note*}

\begin{lemma}\label{lemma2.2}
Let $X,Y$ be subgroups of $G$, $\bar{X} = X  [X,Y]$ and $\bar{Y} = Y [X,Y]$. Then $[X,Y]$, $\bar{X}$, and $\bar{Y}$ are all normal in $\left<X, Y\right>$ and $\left<X, Y\right> =  \bar{X} \bar{Y}$.
\end{lemma}
\begin{proof}
Take $x \in X$, and $y $ and $Z$ in $Y$ in the first equation of (\ref{lemma2.1}). Then $[x,y]^z \in [X,Y]$. Thus $Y$ belongs to the normaliser of $[X,Y]$. So $[X,Y] \nsub \left<X,Y\right>$. The rest is clear.
\end{proof}

\begin{lemma}\label{lemma2.3}
Let $H,K,L$ be normal subgroups of $G$. Then 
$$[H,K] \nsub G \quad \text{ and } \quad [H K, L] = [H,L] [K,L].$$
\end{lemma}

\begin{lemma}\label{lemma2.4}
Let $X = \left<A\right>$ and $Y = \left<B\right>$ be arbitrary subgroups of $G$. Let $C$ be the set of commutators $[a,b]$ with $a \in A$, $b \in B$.  Then $[X,Y] = \left<\left<C^X\right>^Y\right>$.
\end{lemma}
\begin{proof}
For $C \leq [X,Y] \nsub \left<X,Y\right>$ by (\ref{lemma2.2}). Hence $H = \left<\left<C^X\right>^Y\right> \subseteq [X,Y]$. Conversely let $a_{i_1}^{r_1} a_{i_2}^{r_2} \ldots a_{i_n}^{r_n} =x$ be any element of $X$, where each $a_j \in A$. Then for $b \in B$, we have 
$$
b^{-1} x b = \prod_{\alpha =1}^n ( a_{i_\alpha} [a_{i_\alpha},b])^{r^{\alpha}}.
$$ 
Each $[a_{i_\alpha},b] \in C$, each $a_{i_\alpha} \in X$. Hence 
$$
b^{-1}  x b \equiv x \text{ mod } \left<C^X\right>
$$
so that $[x,b] \in H$ for all $x \in X$,  $b \in B$. Let $y =b_{j_1}^{s_1} \ldots b_{j_m}^{s_m}$ be any element of $Y$. Then 
$$
x^{-1} y x= \prod_{\beta=1}^m \pr{b_{j_\beta} [b_{j_\beta}, x]}^{s_\beta}.
$$
Since each $[b_{j_\beta}, x] \in H = H^{Y}$ and each $b_{j_\beta} \in Y$, we have $$x^{-1}  y  x \equiv  y\text{ mod } H$$ or $[y,x] \in H$.
\end{proof}

From (\ref{lemma2.3}) we obtain
\begin{thm}[Fitting]\label{theorem2.5}
If $H$ and $K$ are normal nilpotent subgroups of $G$ of classes $c$ and $d$, then $H  K$ is nilpotent of class at most $c + d$.
\end{thm}
\begin{proof}
For $\gamma_n(HK) = [H K, \ldots, H  K]$ with $n$ terms of the form \begin{equation}\label{fitting_eqn_1}
=\prod [L_1, L_2, \ldots, L_n]
\end{equation} 
where each $L_i$ is either $H$ or $K$. Since $H \nsub G$, we have also $\gamma_r(H) \nsub G$ for all $r$ and so $[\gamma_r(H), K] \leq \gamma_r(H)$. Thus if just $r$ of the terms $L_1, \hdots, L_n$ are equal to $H$ and the remaining $n-r$ equal to $K$, we have 
$$
[L_1, L_2, \ldots L_n] \subseteq \gamma_r(H) \cap \gamma_{n-r}(K).
$$
Choose $n = c + d +1$. Then for each factor in (\ref{fitting_eqn_1}), either $r>c$ or $n-r>d$. In either case that factor is $1$. So $\gamma_{c+d+1}(HK) = 1$.
\end{proof}

\begin{lemma}\label{lemma2.6}
Let $G$ be nilpotent of class $c$ and let $H$ be a proper subgroup of $G$. Define $H_0 = H$ and, inductively, $H_{i+1}$ to be the normaliser of $H_i$ in $G$. Then $$H = H_0 < H_1 < \ldots < H_r = G$$ for some $r \leq c$.
\end{lemma}
\begin{proof}
For $H_i \geq Z_i = \zeta_i(G)
$ by induction on $i$ and $Z_c = G.$\end{proof}

\begin{defn*} A subgroup $H$ of a group $G$ is said to be subnormal in $G$ if there exists a series of subgroups $$H = H_0 \nsub H_1 \nsub H_2 \nsub \ldots \nsub H_r = G.$$ Thus all subgroups of nilpotent groups are subnormal.
\end{defn*}

\begin{cor*}[\textbf{1}] If $H$ is a proper subgroup of the nilpotent group $G$, then $G' H$ is also a proper subgroup of $G$.
\end{cor*}
\begin{proof}
For $H$ is subnormal in $G$. Hence the normal closure $K$ of $H$ in $G$ is also a proper subgroup of $G$. $G / K$ is non-trivial nilpotent group, and so its derived group, which is $KG' / K$, is a proper subgroup of $G / K$. Thus $G' K < G$; a fortiori $G' H$ is a proper subgroup of $G$.
\end{proof}

Suppose that $H$ and $K$ are subgroups of $G$ and that $|G:H| = m$ is finite. In general one can only conclude that $[K: H \cap K] \leq m$; and, if $|G:K| = n$ is also finite, that $[G: H \cap K] \leq mn$. But if $H$ is subnormal in $G$, then $[K:H \cap K]$ divides $m$ and $[G: H \cap K]$ divides $mn$. Thus, we may state the following.

\begin{cor*}[\textbf{2}] If $G$ is a nilpotent group and $H,K$ are subgroups of $G$ such that $|G:H| = m$ is finite, then $[K : H \cap K]$ divides $m$. If $|G:K| = n$ is also finite, then $[G: H \cap K]$ divides $mn$. If $$L = \bigcap_{x \in G} H^x,$$ then $[G:L]$ divides some power of $m$.
\end{cor*}

\begin{thm}\label{theorem2.7}[Burnside-Wielandt]
Let $G$ be a finite group. The following are equivalent.
\begin{itemize} 
\item[(i)] Every maximal subgroup of $G$ is normal in $G$.
\item[(ii)] $G$ is the direct product of its Sylow subgroups.
\item[(iii)] $G$ is nilpotent.
\end{itemize}
\end{thm}
\begin{proof}
For let $S$ be a Sylow subgroup of $G$, and $N$ its normaliser in $G$. If $N < G$, we can choose a maximal subgroup $M$ of $G$ which contains $N$. A simple application of Sylow's Theorem shows that $M$ is its own normalizer in $G$. Thus \textit{(i)} $\Rightarrow$ \textit{(ii)}. By (\ref{lemma1.4}), \textit{(ii)} $\Rightarrow$ \textit{(iii)}. By \ref{lemma2.6}, \textit{(iii)} $\Rightarrow$ \textit{(i)}.
\end{proof}

\begin{defn*}  Let $\mathcal{P}$ be any property of groups. A group $G$ is called \textbf{locally-$\mathcal{P}$} if every finitely generated subgroup of $G$ is $\mathcal{P}$. \footnote{It more common now to say $G$ is locally-$\mathcal{P}$ if every finitely generated subgroup is contained in a subgroup of $G$ that is $\mathcal{P}$.}
\end{defn*}
\begin{thm}\label{2.8}
$\:$\\
\vspace{-3.5mm}
\begin{itemize}
\item[(i)] Every maximal subgroup of a locally nilpotent group $G$ is normal in $G$. (McLain) 
\item[(ii)] If $H$ and $K$ are normal locally nilpotent subgroups of $G$, then $H K$ is also locally nilpotent (Hirsch).
\item[(iii)] If ever proper subgroup of $G$ is distinct from its normalizer in $G$, then $G$ is locally nilpotent. (Plotkin) \footnote{Heineken and Mohamad have shown recently that there exist groups with this normalizer condition but with trivial centre (J. Algebra, 10 (1968) 368-376).}
\end{itemize}
\end{thm}
\begin{proof}
(i) Let $G$ be locally nilpotent. Suppose if possible that $M$ is a maximal subgroup of $G$ which s not normal in $G$. Then $M  \not \geq G'$ and we may choose an element $x$ in $G'$ which  is not in $M$. Hence $\left<M, x\right> = G$. Since $x \in G'$, there exist elements $y_i, z_i$ in $G$ such that $$x = \prod_{i=1}^n [y_i, z_i].$$ Each $y_i$ and $z_i$ may be expressed in terms of $x$ together with elements of $M$. Let $u_1, \ldots, u_r$ be the elements of $M$ involved in these expressions and let $H = \left<x, u_1, \ldots, u_r\right>$. Then $H$ is nilpotent and each $y_i$ and $z_i$ is in $H$. Hence $x \in H'$. But $K = \left<u_1, u_2, \ldots, u_r\right>$ does not contain $x$ since $K \leq M$. Hence $H<K$. Since $H$ is finitely generated, there exists a maximal subgroup $N$ of $H$ which contains $K$. By (\ref{lemma2.6}), $N \nsub H$ since $H$ is nilpotent. Hence $H/N$ is cyclic of order a prime, $N \geq H'$, $x \in N$, $N$ contains $\left<x,K\right> = H$, a contradiction.

(ii) Let $H$ and $K$ be normal in $G$ and locally nilpotent. Let $a_1b_1, a_2b_2, \ldots, a_nb_n$ be any finite set of elements of $HK$, where each $a_i \in H$ and each $b_i \in K$. Let $X = \left<a_1, \ldots, a_n\right>$ and $Y = \left<b_1, \ldots, b_n\right>$. We have to prove that $\left<a_1 b_1, \ldots, a_nb_n\right>$ is nilpotent and it will be sufficient to prove $\left<X,Y\right> = L$ is nilpotent. Let $C$ be the set of all commutators $[a_i,b_j]$. Since $H$ and $K$ are normal in $G$, $[H,K] \leq H \cap K$. Hence $\left<C,X\right>$ is a finitely generated subgroup of $H$ and therefore nilpotent. By (\ref{lemma1.9}), the subgroup $\left<C^X\right>$ of $\left<C, X\right>$ is finitely generated. Also $\left<C^X\right> \leq H \cap K$. Hence $\left<C^X, Y\right>$ is a finitely generated subgroup of $K$ and its subgroup $\left<\left<C^X\right>^Y\right> = [X,Y]$ by (\ref{lemma2.4}) is also finitely generated. Hence $\bar{X} = X[X,Y]$ is a finitely generated subgroup of $H$ and therefore nilpotent. Similarly $\bar{Y} = Y[X,Y]$ is nilpotent. But by (\ref{lemma2.2}), $L = \left<X,Y\right> = \bar{X} \bar{Y}$ and $\bar{X}, \bar{Y}$ are nilpotent in $L$. Hence $L$ is nilpotent by (\ref{theorem2.5}).

(iii) Let $G$ be any group. By Zorn's Lemma, every element $x$ of $G$ is contained in some maximal locally nilpotent subgroup of $M$ of $G$. Let $N$ be the normaliser of $M$ and let $y$ belong to the normaliser of $N$. Then $M$ and $M^y$ are normal locally nilpotent subgroups of $N$. By Hirsch's theorem, their product is also locally nilpotent. By the maximality of $M$, it follows that $M = M^y$, $y \in N$, $N$ is it own normaliser in $G$. If therefore $G$ has the property that each of its proper subgroups is distinct from its normaliser in $G$, we must have $N = G$, $M \nsub G$. In such a group every maximal locally nilpotent subgroup is normal. But by Hirsch's theorem, $M$ is then unique. Since every $x \in G$ belongs to some such $M$, we have $M=G$; $G$ is itself locally nilpotent.
\end{proof}
\begin{thm}\label{theorem2.9}
Let $M$ be a minimal normal subgroup of the group $G$.
\begin{itemize}
\item[(a)] If $G$ is locally soluble, then $M$ is abelian. 
\item[(b)] If $G$ is locally nilpotent, then $M$ lies in the centre of $G$ (when $M$ is finite of prime order).
\end{itemize}
\end{thm}
\begin{proof}
(a) Assume, if possible, that $M$ is not abelian, and let $a$ and $b$ be elements of $M$ such that $c= [a,b] \neq 1$. By the minimality of $M$, $\left<c^G\right> = M$, so for suitable $x_1, \ldots, x_n$ in $G$, $a$ and $b$ lie $\left<c^{x_1}, c^{x_2}, \ldots, c^{x_n}\right>$. Let $H = \left<c, x_1, \ldots, x_n\right>$ and $K = \left<c^H\right>$. Now $c \in K'$ as both $a$ and $b$ are in $K$, and also $K' \nsub H$. We conclude $K' = K$, which is a contradiction since $H$ is by hypothesis a soluble group. Hence $M$ is Abelian.

(b) Assume, if possible, that the result is false. Then for some $a$ in $M$, $b$ in $G$, we have $c = [a,b] \neq 1$. By the minimality of $M$, $\left<c^G\right> = M$, so for suitable $x_1, \ldots, x_n$ in $G$, $a$ lies in $\left<c^{x_1}, c^{x_2}, \ldots, c^{x_n}\right>$. Let $H = \left<b, c, x_1, \ldots, x_n\right>$ and $K = \left<c^H\right>$. Since $a \in K$, $b \in H$ it follows that $c \in [K,H]$. As $[K,H]$ is normal in $H$, $[K,H] = K$. This is a contradiction since $H$ is by hypothesis a nilpotent group. Hence $M$ lies in the centre of $G$. \end{proof}

\begin{cor*}
A simple, locally soluble group must be cyclic of order  a prime.
\end{cor*}

\section{Nilpotent Groups of Automorphisms}
\begin{lemma}\label{lemma3.1}
$$[x, y^{-1}, z]^y[y, z^{-1},x]^z[z,x^{-1},y]^x = 1.$$
\end{lemma}
\begin{proof}
Let $u = x z  x^{-1}  y  x$ and let $v,w$ be obtained from $u$ by cyclic permutation of $x, y, z$. Then 
$$
[x, y^{-1}, z]^{y} = y^{-1}  [y^{-1}, x] z^{-1} [x, y^{-1}] zy = x^{-1} y^{-1}  x  z^{-1}  x^{-1}  y  x  y^{-1}  z 
 y = u^{-1}v = 1;
 $$ etcetera.
\end{proof}

\begin{lemma}\label{lemma3.2}
Let $X, Y, Z$ be subgroups of $G$. Let $X^* = [Y, Z, X]$, $Y^*  = [Z, X, Y]$, and $Z^* = [X, Y, Z]$. If $N \nsub G$ and $X^*, Y^* \leq N$, then $Z^* \leq N$.
\end{lemma}

\begin{cor*}
If $X$, $Y$, $Z$ are all normal in $G$, then $Z^* \leq X^*Y^*$.
\end{cor*}
\begin{proof}
For suppose $X^{*}  Y^{*} \leq N \nsub G$. Let $x \in X$, $y \in Y$, $z \in Z$ in (\ref{lemma3.1}). The last two factors on the left both lie in $N$. Hence $[x,y^{-1},z] \in N$; so every element $z$ of $Z$ commutes $\text{mod } N$ with every generator of $[x,y^{-1}]$ of $[X,Y]$. Thus $Z^{*} \leq N$.
\end{proof}

Also $[x,y,z^x]  [z,x,y^z]  [y,z,x^y] = 1$. In metabelian groups, we can write 
$$
[x,y,z] \cdot [z,x,y] \cdot [y,z,x] = 1.
$$

\begin{thm}\label{theorem3.3}
Let $H$ and $K$ be subgroups of $G$. Let $$H=H_0 \geq H_1 \geq \ldots$$ be a series of subgroups, all normal in $H$, and such that $[H_i,K] \leq H_{i+1}$ for each $i = 0, 1, 2, \ldots \: $. Define $K = K_1$ and let $K_j$ consist of all elements $x$ of $K$ such that $[H_i, x] \leq H_{i+j}$ for all $i$. Then $[K_j, K_\ell] \leq K_{j + \ell}$ for all $j, \ell = 1,2, \ldots$ and $[H_i, \gamma_j(K)] \leq H_{i+j}$ for all $i,j$.
\end{thm}
\begin{proof}
For by definition of $K_j$ and $K_\ell$, both $[H_i,K_j,K_\ell]$ and $[H_i,K_\ell,K_j]$ are contained in $H_{i+j+\ell}$. But $H_{i+j+\ell} \nsub \left<H,K\right>$ so that (\ref{lemma3.2}) applies and gives $[H_i, [K_j,K_\ell]] \leq H_{i+j+\ell}$ and therefore $[K_j, K_\ell] \leq K_{j+\ell}$. Thus $K = K_1 \geq K_2 \geq \ldots$ is a central series of $K$ and $\gamma_j(K) \leq K_j$ so that $[H_i, \gamma_j(K)] \leq H_{i+j}$.
\end{proof}

\begin{thm}\label{theorem3.4}
For any group $G$, we have 
\begin{itemize}
\item[(i)] $[\Gamma_i, \Gamma_j] \leq \Gamma_{i+j}$
\item[(ii)] If $i \geq j$, $[Z_i, \Gamma_j] \leq Z_{i-j}$ and in particular $[Z_i, \Gamma_i] = 1$
\item[(iii)] Defining the \underline{derived series} $$G \geq G' \geq G'' \geq \ldots \geq G^{(k)} \geq$$ of $G$ by $G^{(k+1)} = [G^{(k)}, G^{(k)}]$, we have that $G^{(k)} \leq \Gamma_{2^k}$.
\end{itemize}
\end{thm}
\begin{proof}
These results are all corollaries of (\ref{theorem3.3}). For \textit{(i)}, take $H_i = \Gamma_{i+1} = \gamma_{i+1}(G)$ and $K = G$. For \textit{(ii)}, take $H_j = Z_{i-j}$ for $j = 0,1, \ldots,i$ and again $K = G$. \textit{(iii)} follows from \textit{(i)} by induction on $k$.
\end{proof}

\begin{lemma}\label{3.5}
Let $$G = G_0 \geq G_1 \geq \ldots \geq G_r = 1$$ be a series of normal subgroups of $G$. Let $A$ be the group of all automorphisms of $G$ which leave each $G_i$ invariant and transform each $G_{i-1} / G_i$ identically. Then both $A$ and $[G,A]$ are nilpotent of class $< r$.
\end{lemma}
\begin{proof}
Here $[G,A]$ is the subgroup of $G$ generated by $x^{-1}x^a = [x,a]$ with $x$ in $G$ and $a$ in $A$. We think of $G$ as replaced by its regular representation, so that $A$ and $G$ become subgroups of the holomorph of $G$. Then $[G_{i-1}, A] \leq G_i$ for each $i = 1,2, \ldots, r$. By (\ref{theorem3.3}), $[G, \gamma_r(A)] \leq G_r = 1$, so that $\gamma_r(A) = 1$. Also, since $G_{i-1} \nsub G$, we have 
$$
[G_{i-1}, G, A] \leq [G_{i-1}, A] \leq G_i.
$$
But $G_i \nsub \left<A, G\right>$. So, by (\ref{lemma3.2}), $[G_{i-1}, [G,A]] \leq G_i$ for each $i$. By (\ref{theorem3.3}) again, $[G_1, \gamma_{r-1}([G,A])] = 1$. But $[G,A] \leq G_1.$ Hence $\gamma_r([G,A]) = 1$.
\end{proof}

\begin{cor*}Let $\nu$ be any ring (associative and with a unit element). Let $T_n(\nu)$ be the group of all \underline{unitriangular} $n \times n$ matrices $x = (x_{ij})$ with coefficients in $\nu$ ($x_{ij=0}$ for $i > j$ and $x_{11} = x_{22} = \ldots = x_{nn} = 1)$. Then $T_n(\nu)$ is nilpotent of class $n-1$.
\end{cor*}

\begin{thm}\label{theorem3.6}
Let $\mathfrak{k}$ be an ideal of $\nu$ such that $\mathfrak{k}^n = 0.$ Let $K_i = 1 + \mathfrak{k}^i (i = 1,2, \ldots).$ Then each $K_i$ is a multiplicative group; $[K_i, K_j[ \leq K_{i+j};$ $K = K_1$ is nilpotent of class $< n$ and $\gamma_i(K) \leq K_i.$
\end{thm}
\begin{proof}
Define $H_k$ to consist of all translations $\tau_a: x \to x + a$ of $\nu$ such that $a \in \mathfrak{k}^k$. Interpret $\mathfrak{k}^0 = \nu.$ We may consider $K_i$ to consist of all multiplications $\mu_{1+u} \colon x \to x(1 + u)$ of $\nu$ for which $u \in \mathfrak{k}^i$. Since $u^n = 0$, $\mu_{1+u}$ has inverse $\mu_{1+v}$ where
$$
v = -u + u^2 - u^3 + \ldots \pm u^{n-1}
$$
also belongs to $\mathfrak{k}^i.$ Also 
$$
\mu_{1+u} \mu_{1+u'} = \mu_{1 + u + u' + uu'}.
$$
Thus $K_i$ is a group. Also  
$$
[\tau_a, \mu_{1+u}] = \tau_{au}.
$$
Thus, with $K = K_1,$ we have 
$$
[H_{i-1}, K] \leq H_i
$$ for each $i$. Again, since $\tau_1 \in H = H_0$ and 
$$
[\tau_1, \mu_{1+u}] = \tau_u,
$$ we see that $K_j$ consists of precisely those elements $x$ of $K$ such that 
$$
[H_i, x] \leq H_{i+j}
$$
for all $i$. (\ref{theorem3.3}) now gives the required result.
\end{proof}

\begin{lemma}\label{lemma3.7}
Let $H$ and $K$ be subgroups of $G$ such that $[H,K] \leq H'.$ Then 
$$
[\gamma_i(H), \gamma_j(K)] \leq \gamma_{i+j}(H)
$$
for all $i,j = 1,2, \ldots.$
\end{lemma}
\begin{proof}
Proof by double induction.
\end{proof}

\begin{cor*} Let $H$ be nilpotent of class $c$ and let $K$ be the group of all automorphisms of $H$ which transform $H/H'$ identically. Then $\gamma_j(K)$ transforms each of the groups $\gamma_i(H) / \gamma_{i+j}(H)$ identically and $K$ is nilpotent of class $<c.$
\end{cor*}

\begin{lemma_bis}\label{lemma3.8}
Let $H \nsub K$. If both $H$ and $K/H'$ are nilpotent, so is $K$.
\end{lemma_bis}

\begin{thm}\label{theorem3.8}
Let $H$ and $K$ be arbitrary subgroups of $G$. Define $H_0 = H$ and inductively $H_{i+1} = [H_i, K].$ Then $H_r = 1$ implies $[H, \gamma_{1 + {r \choose 2}}(K)] = 1$ 
\end{thm}
\begin{proof}
The statement is clear for $r = 0,1$. Suppose $r>1$ and let $L = \left<H,K\right>.$ Then $H_1 \nsub L$ by (\ref{lemma2.2}), so that $H_1 \geq H_2 \geq \ldots \geq H_r = 1.$ Let $C$ be the centralizer of $H_1$ in $K$. Then $K / C \cong A$, the group of automorphisms of $H_1$ induced by $K$. Think of $A$ and $H_1$ as subgroups of the holomorph of $H_1.$ Then 
$$
[H_i, A] = H_{i+1} \quad (i=1, 2, \ldots, r-1).
$$ 
By induction on $r$, we may suppose $[H_1, \gamma_{1 + {r-1 \choose 2}}(A)] = 1$, so that $\gamma_{1 + {r-1 \choose 2}}(A) =1$ or, equivalently $\gamma_{1 + {r-1 \choose 2}}(K) \leq C$. Define $C = C_1$ and inductively $C_{i+1} = [C_i, K].$ We need only show that $[H, C_r] = 1.$ Since $C \leq K,$ we have $[H, C_1] \leq H_1.$ Suppose that for some $i < r$ we have proved that $[H, C_i] \leq H_i$. Let $y \in K$, $x \in C_i$, and $z \in H$. Then $[y, z^{-1}] \in H_1$ and hence $[y, z^{-1},x] = 1$ since $C_i \leq C.$ By (\ref{lemma3.1}), 
$$
[x, y^{-1},z][z, x^{-1},y]^{xy^{-1}} = 1.$$
But $[z,x^{-1}] \in H_i$ by the induction hypothesis. So $[z,x^{-1},y]^{xy^{-1}}$ belongs to $[H_i, K]^{xy^{-1}} = [H_i, K]$ by (\ref{lemma2.2}) since $xy^{-1} \in K.$ But $[H_i, K] = H_{i+1}.$ Hence $[x,y^{-1},z] \in H_{i+1}$. Here $[x,y^{-1}]$ is a typical generator of $C_{i+1}.$ Let $t = u_{j_1}^{r_1}u_{j_2}^{r_2} \ldots u_{j_n}^{r_n}$ be any element of $C_{i+1},$ expressed in terms of such generators $u_1, u_2, \ldots \: .$ Hence each $[u_j, z] \in H_{i+1}$ and $$
z^{-1}tz = \prod_{\alpha}(u_{j_\alpha} [u_{j_{\alpha}},z])^{r_{\alpha}}.
$$
Each $u_j$ belongs to $C$ and therefore commutes with every element of $H_1 \geq H_{i+1}.$ So 
$$
z^{-1}tz = t \prod_{\alpha}[u_j, z]^{r_\alpha}
$$ and $[t,z] \in H_{i+1}.$ So finally $[H, C_{i+1}] \leq H_{i+1}$ and all is proved. 
\end{proof}

\begin{cor*} Let $G = G_0 \geq G_1 \geq \ldots \geq G_r = 1$ be a chain of subgroups of $G$ and let $A$ be a group of automorphisms of $G$ such that $[G_{i-1}, A] \leq G_i$ for each $i=1, 2, \ldots, r$. Then $A$ is nilpotent of class at most ${r \choose 2}.$
\end{cor*}

\section{Isolators}
\begin{lemma}\label{lemma4.1}
Let $X$ and $Y$ be subgroups of $G$ such that $[X,Y,G] = 1$. Then, for $x \in X$ and $y \in Y$, the function $[x,y]$ is homomorphic in both arguments.
\end{lemma}
\begin{proof}
This is a corollary of (\ref{lemma2.1}).
\end{proof}

An example: we may take $X= G$ and $Y = Z_2 = \zeta_2(G).$ Given $y \in Z_2$, the map $x \to [x,y]$ is homomorphic for $x \in G$. Hence $G/ C_G(y) \cong [G,y],$ the image group, which being contained in $Z_1$ is Abelian. Hence $G' \leq C_G(y)$ and $[Z_2, G'] = 1,$ an important special case of (\ref{theorem3.4}(ii)). In particular, in a perfect group $G = G',$ we have $Z_1 = Z_2 = \ldots \:.$

\begin{lemma}\label{lemma4.2}
Let $Z_r = \zeta_r(G)$. If the centre $Z_1$ of $G$ is of exponent $m > 0$ (i.e. if $x^m = 1$ for all $x \in Z_1$), then $Z_{r+1} / Z_r$ is of exponent $m$ and $Z_r$ is of exponent $m^r$ for all $r = 0,1,2, \ldots \:.$
\end{lemma}
\begin{proof}
Clearly, it is enough to show that $Z_2 / Z_1$ is of exponent $m$, i.e. that $[x^m,y] = 1$ for all $x \in Z_2$ and $y \in G$. But $[Z_2, G, G] = 1$ and so $[x^m,y] = [x,y]^m$ by (\ref{lemma4.1}). Since $[x,y] \in Z_1,$ the result follows.
\end{proof}
\begin{cor*} A finitely generated nilpotent group $G$ whose centre $Z_1$ is a $p$-group is itself a finite $p$-group.
\end{cor*}
\begin{proof}
For $Z_1$ is finitely generated by (\ref{lemma1.9}), hence of exponent $m= p^k$ for some $k$. Also $G = Z_r$ for some $r$. All the $Z_i / Z_{i-1}$ $(i=1, \ldots, r)$ are finitely generated Abelian groups, of exponent $m$ by (\ref{lemma4.2}), hence they are finite $p$-groups. So therefore is $G$.
\end{proof}

\begin{lemma}\label{lemma4.3}
Let $H$ be a subgroup of the finitely generated nilpotent group $G$ and let $K$ be a subgroup of $H$ such that $[H:K]$ is either infinite or divisible by the given prime $p$. Suppose further that $[MH: MK]$ is finite and prime to $p$ for every normal subgroup $M \neq 1$ in $G$. Then $G$ is in fact a finite $p$-group.
\end{lemma}
Note that $M \cap H = 1$ and $M \leq K$ are alike impossible if $1 \neq M \nsub G;$ for each implies that $|MH: MK| = |H:K|.$
\begin{proof}
By the preceding corollary, it is enough to show that the centre $Z_1$ of $G$ is a $p$-group. Taking $M = \left<a\right>$ where $1 \neq a \in Z_1,$ suppose first that $M$ could be infinite. Since $M \cap H \neq 1$, we may suppose $M \leq H.$ But $M \cap K \nsub G$ and so $M \cap K = 1.$ If $M_1 = \left<a^p\right>,$ this gives 
$$
|M_1H : M_1K| = |H:MK|\cdot |MK: M_1K| = [H:MK] \cdot p
$$ which is not prime to $p$ although $1 \neq M_1 \nsub G.$ Hence $M$ cannot be infinite, and so $Z_1$ is periodic. Next, let $M$ be of order a prime. Then $M \leq H$ and $M \cap K = 1$ as before, while 
$$
|H:K| = |H:MK|\cdot|MK:K| = |MH: MK| \cdot|M|.
$$ 
By hypothesis, $[MH: MK]$ is finite but prime to $p$. Hence $[H:K]$ is finite and therefore divisible by $p$. Thus $|M| = p$ and $Z_1$ is a $p$-group as required.
\end{proof}

Now let $\theta(x_1, \ldots, x_n)$ be any word. If $H_1, \ldots, H_n$ are subgroups of any group $G$, we define the generalized verbal subgroup $\theta(H_1, \ldots, H_n)$ to be the group generated by all elements of the form $\theta(h_1, \ldots, h_n)$ with $h_i \in H_i$ for each $i=1, 2, \ldots, n$. Note that if $x \to x^*$ is any homomorphic mapping of $G$, we have $$\theta(H_1^*, \ldots, H_n^*) = (\theta(H_1, \ldots, H_n))^*.$$

\begin{thm}\label{4.4}
Let $G$ be a finitely generated nilpotent group; let $H = \theta(H_1, \ldots, H_n)$ where the $H_i$ are subgroups of $G$ and let $K = \theta(K_1, \ldots, K_n)$ where, for each $i$, $K_i$ is a subgroup of finite index $m_i$ in $H_i$. Then $[H:K]$ is also finite and divides some power of $m=m_1m_2 \ldots m_n.$
\end{thm}
\begin{proof}
Suppose not. Since $G$ satisfies \textbf{Max}, we can choose $M \nsub G$ to be maximal subject to the condition that $[MH: MK]$ is either infinite or else has a prime divisor $p$ which does not $m$. Let $x^* = xM$ for $x \in G$. Applying (\ref{lemma4.3}) shows that $G^* = G/M$ is in fact a finite $p$-group. But 
$$
|H_i^*: K_i^*| = |H_i : H_i(M \cap H_i)|
$$ 
divides $m_i$. Since $(m,p) = 1,$ it follows that $H_i^* = K_i^*$ for all $i = 1,2, \ldots, n$. Hence 
$$
H^* = \theta(H_1^*, \ldots, H_n^*) = \theta(K_1^*, \ldots, K_n^*) = K^*,
$$ i.e. $MH = MK,$ a contradiction. 
\end{proof}

\begin{thm}\label{theorem4.5}
Let $G$ be a locally nilpotent group. For any subgroup $H$ of $G$, write $J_{\bar{\omega}, G}(H) = \bar{H}$. \footnote{Let $\bar{\omega}$ denote a set of (rational) primes. We make the following definitions:
    \begin{itemize}
        \item[(i)] A positive integer $n$ is $\bar{\omega}$-number if either $n=1$ or all the prime divisor of $n$ are in $\bar{\omega}.$ 
        \item[(ii)] An element $x$ of a group $G$ is a $\bar{\omega}$-element if it is periodic and has order a $\bar{\omega}$-number.
        \item[(iii)] A group $G$ is a $\bar{\omega}$-group if all its elements are $\bar{\omega}$-elements.
        \item[(iv)] A group $G$ is $\bar{\omega}$-torsion-free if the identity is the only $\bar{\omega}$-element of $G$.
        \item[(v)] Let $H, K$ be subgroups of a group $G$. We say that $H$ is $\bar{\omega}$-equivalent to $K(H \sim_{\bar{\omega}} K)$ for any $x \in H, y\in K$ there exist $\bar{\omega}$-numbers $m$ and $n$ such that $x^m \in K, y^n \in H.$
        \item[(vi)] A subgroup $H$ of a group $G$ is $\bar{\omega}$-isolated in $G$ if $x \in G$ and $x^n \in H$ for some $\bar{\omega}$-number $m$ imply $x \in H.$
        \item[(vii)] Let $H$ be a subgroup of $G$, then $J_{\bar{\omega},G}(H)$, the $\bar{\omega}-$isolator of $H$ in $G$, is the intersection of all subgroups containing $H$ which are $\bar{\omega}$-isolated in $G$.
        \item[(viii)] $\bar{\omega}'$ denotes the set of primes not in $\bar{\omega}.$
    \end{itemize}} Then 
\begin{itemize}
\item[(a)] $H \sim_{\bar{\omega}} \bar{H}$. 
\item[(b)] The elements $x$ of $G$ such that $x^r \in H$ for some $\bar{\omega}$-number $r$ form a subgroup (which is therefore $\bar{\omega}$-isolated in $G$, hence coincides with $\bar{H}).$
\item[(c)] Let $H_\lambda, K_\lambda$ $(\lambda \in \Lambda)$ be subgroups of $G$ such that $H_\lambda \sim_{\bar{\omega}} K_\lambda$ for all $\lambda$, and let $A = \left< H_\lambda ; \lambda \in \Lambda\right>$ be the group generated by the $H_\lambda$ and similarly let $B = \left<K_\lambda ; \lambda \in \Lambda\right>$. Then $A \sim_{\bar{\omega}} B$.
\item[(d)] The set $R_{\bar{\omega}}$ of all $\bar{\omega}$-elements of $G$ is a characteristic subgroup (the $\bar{\omega}$-radical) of $G$ and $G/ R_{\bar{\omega}}$ is $\bar{\omega}$-torsion-free.
\item[(e)] If $G$ is finitely generated, then $[\bar{H}:H]$ is a $\bar{\omega}$-number.
\end{itemize}
\end{thm}
\begin{proof}
The statements (a) and (b) are equivalent, and (b) follows from the corollary to Theorem \ref{4.4}. For if $g_1, g_2$ are elements of $G$ such that $g_1^{n_1}$ and $g_2^{n_2}$ lie in $H$, where $n_1$ and $n_2$ are $\bar{\omega}$-numbers, then $[\left<g_1, g_2\right> : \left<g_1^{n_1}, g_2^{n_2}\right>] = n$ is also a $\bar{\omega}$-number. Since $\left<g_1^{n_1}, g_2^{n_2}\right>$ is contained in $H$ and is subnormal in $\left<g_1, g_2\right>$, it follows that $(g_1 g_2^{-1})^n \in H.$

(c) Since $H_\lambda \sim_{\bar{\omega}} K_\lambda,$ we have $H_\lambda \leq \bar{K}_\lambda.$ Hence $A \leq \bar{B}.$ Similarly $B \leq \bar{A};$ and so $A \sim_{\bar{\omega}} B.$

(d) Evidently $R_{\bar{\omega}} = J_{\bar{\omega}, G}(1)$ by \textit{(b)}. So $R_{\bar{\omega}}$ is $\bar{\omega}$-isolated characteristic subgroup of $G$. If $K \nsub G$, then clearly $G/K$ is $\bar{\omega}$-torsion-free if and only if $K$ is $\bar{\omega}$-isolated in $G$.

(e) We may suppose by (\ref{lemma2.6}) that $H = H_0 \nsub H_1 \nsub \ldots \nsub H_r = \bar{H}$ and each of the factors $H_i / H_{i-1}$ is a finitely generated nilpotent $\bar{\omega}$-group. By (\ref{theorem1.10}), then, each $H_i / H_{i-1}$ is a finite $\bar{\omega}$-group. Thus $[\bar{H}: H]$ is a $\bar{\omega}$-number, as required.
\end{proof}

\begin{thm}\label{4.6}
Let $G$ be a locally nilpotent group and let $\theta$ be any word in $n$ variables. Let $H_i, K_i$ be subgroups of $G$ such that $H_i \sim_{\bar{\omega}} K_i$ for each $i = 1,2, \ldots, n$. Then
$$
H = \theta(H_1, \ldots, H_n) \sim_{\bar{\omega}} K = \theta(K_1, \ldots, K_n).
$$
\end{thm}
\begin{proof}
It is sufficient to show, by (\ref{theorem4.5}), that $H \leq \bar{K}.$ So we must prove that for any $u = \theta(h_1, \ldots, h_n),$ where $h_i \in H_i,$ $u^r$ lies in $K$ for some $\bar{\omega}$-number $r$.

Let $P = \left<h_1, h_2, \ldots, h_n\right>,$ which is finitely generated nilpotent. Now $h_i^{n_i} \in K$ for certain $\bar{\omega}$-numbers $n_i$ $(i = 1,2, \ldots, n)$. Define $P_i = \left<h_i\right>$, $Q_i = \left<h_i^{n_i}\right>,$ then $|P_i : Q_i| = m_i$ is a $\bar{\omega}$-number for each $i= 1, \ldots, n.$ It follows from (\ref{4.4}) that $$|\theta(P_1, \ldots, P_n) : \theta(Q_1, \ldots, Q_n)|$$ is a $\bar{\omega}$-number. $u$ is in $\theta(P_1, \ldots, P_n)$ and $\theta(Q_1, \ldots, Q_n)$ is subnormal in $\theta(P_1, \ldots, P_n)$, so $u^r = \theta(Q_1, \ldots, Q_n)$ for some $\bar{\omega}$-number $r$. This completes the proof since $\theta(M_1, \ldots, M_n) \leq K.$
\end{proof}

\noindent \textbf{\underline{Example}} Let $\theta = [x_1, x_2]$. Then we may deduce that $[H_1, H_2] \sim_{\bar{\omega}} [K_1, K_2]$ whenever $H_1 \sim_{\bar{\omega}} K_1$ and $H_2 \sim_{\bar{\omega}} K_2.$

\begin{cor*} If $U$ and $V$ are arbitrary subgroups of $G$, then by (\ref{theorem4.5}(a)), 
$$
[\bar{U}, \bar{V}] \leq \overline{[U,V]}.
$$
\end{cor*}

A more elementary proof of part (d) of the following lemma was given in (\ref{theorem1.10}).

\begin{lemma}\label{4.7}
Let $G$ be a locally nilpotent group, $S$ a normal subgroup of $G$, $R$ a subgroup of $G$ containing $S$. Let $C = C_G(R/S).$ Then we have 
\begin{itemize}
\item[(a)] $\bar{C} \leq C_G(\bar{R} / \bar{S})$.
\item[(b)] $S = \bar{S}$ implies $C = \bar{C}$.
\item[(c)] $\bar{S} \geq R$ implies $C$ is $\bar{\omega}'-$isolated in $G$. 
\item[(d)] Periodic elements of coprime order in $G$ commute.
\item[(e)] If, in addition, $R \nsub G$ and $R/S$ is a finite $\bar{\omega}$-group, then $G/C$ is also a finite $\bar{\omega}$-group.
\end{itemize}
\end{lemma}
\begin{proof}
(a) By (\ref{4.6}) corollary, $[\bar{C}, \bar{R}] \leq \overline{[C,R]}.$ But $[C,R] \leq S.$ 

(b) If $\bar{S} = S$, $[\bar{C}, \bar{R}] \leq S$. Hence $[\bar{C}, R] \leq S,$ so that $\bar{C} = C.$

(c) Let $x^m \in C$ for some $\bar{\omega}'$-number $m$. If $X = \left<x\right>,$ $Y = \left<x^m\right>$, we have, by (\ref{4.6}), $[X, R] \sim_{\bar{\omega}} [Y,R]$. Hence, $[X,R] \leq J_{\bar{\omega}', G}(S)$ as $[Y,R] \leq S.$ But by hypothesis, $R \sim_{\bar{\omega}} S$, so, again by (\ref{4.6}), $[X,R] \sim_{\bar{\omega}} [X,S]$. As $S \nsub G,$ we conclude that $[X,R] \leq J_{\bar{\omega}, G}(S).$ Hence $[X,R] \leq S$ and so $x \in C.$

(d) Let $x^m =1, y^n =1$ and $\bar{\omega}$ such that $m \in \bar{\omega}, n \in \bar{\omega}'$. In \textit{(c)} put $S =1,$ $R = \left<x\right>$, so that $C$ is $\bar{\omega}'$-isolated in $G$. But $y^n \in C$ so $y \in C,$ i.e. $x$ and $y$ commute. 

(e) Here $G/C$ is finite as it is a group of automorphisms of $R/S.$ By \textit{(c)}, $C$ is $\bar{\omega}'$-isolated in $G$, so $G/C$ is a finite $\bar{\omega}-$group.
\end{proof}

\begin{lemma}\label{4.8}
Let $G$ be a locally nilpotent group which is also $\bar{\omega}$-torsion-free. Then 
\begin{itemize}
\item[(a)] All centralizer of subsets of $G$ are $\bar{\omega}$-isolated in $G$. 
\item[(b)] All terms of the upper central series are $\bar{\omega}$-isolated in $G$. 
\item[(c)] If $H$ is a subgroup such that $\bar{H} = G,$ then $\zeta_i(H) = \zeta_i(G) \cap H.$
\end{itemize}
\end{lemma}
\begin{proof}
By hypothesis $\bar{1} = 1.$
(a) Let $X$ be a subset of $G$, $K =\left<X\right>$, then $C_G(X) = C_G(K).$ The result now follows from (\ref{4.7}(b)).

(b) We prove $\zeta_i(G) = \overline{\zeta_i(G)}$ by induction on $i$. If $i = 0$ it is true by hypothesis. Assume then that $\zeta_{i-1}(G)$ is $\bar{\omega}$-isolated in $G$. As $\zeta_i(G) = C_G(G / \zeta_{i-1}(G))$ it is $\bar{\omega}$-isolated in $G$ by (\ref{4.7}(b)).

(c) The result is true for $i=0$. Assume that $\zeta_{i-1}(H) = \zeta_{i-1}(G) \cap H.$ Then
\begin{eqnarray*}
[\zeta_i(H), G] &\leq& [\overline{\zeta_i(H)}, \bar{H}] \quad \text{ by hypothesis;}\\
\: &\leq&  \overline{[\zeta_{i}(H), H]} \quad \text{ by (\ref{4.6}) corollary;}\\
\: &\leq& \zeta_{i-1}(H); \\
\: &\leq& \zeta_{i-1}(G) \quad \text{ by induction and part \textit{(b)}}.
\end{eqnarray*}
Thus $\zeta_i(H) \leq \zeta_i(G).$ The reverse inequality is immediate, and the result follows from induction on $i$.
\end{proof}

\begin{lemma}\label{4.9}
Let $H$ be a subgroup of the locally nilpotent group $G$. Then, writing $N= N_G(H)$,
\begin{itemize}
\item[(a)] $\bar{N} \leq N_G(\bar{H})$
\item[(b)] If $G$ is finitely generated, $\bar{N} = N_G(\bar{H})$
\end{itemize}
\end{lemma}
\begin{proof}
(a) By (\ref{theorem4.5}(a)) and (\ref{4.6}), $[\bar{N}, \bar{H}] \leq \overline{[N, H]}$. But from the definition of $N$, $[N,H] \leq H$ where $[\bar{N}, \bar{H}] \leq \bar{H}$. This means that $\bar{N}$ normalises $\bar{H}.$

(b) Let $x \in N_G(\bar{H}):$ we must prove that $x \in \bar{N}.$ It is evident that for this purpose we may assume $G = \left<\bar{N}, x\right>.$ By (\ref{theorem4.5}(e)), $n = [\bar{H}:H]$ is a $\bar{\omega}$-number, and $\bar{H}^n \leq H$ by subnormality. From the first part of this lemma, $\bar{H} \nsub G,$ hence also $\bar{H}^n \nsub G.$ Let $C = C_G(\bar{H} / \bar{H}^n)$, then by (\ref{4.7}(e)), $G/C$ is a finite $\bar{\omega}$-group and so $G = \bar{C}$. But $C$ normalizes $H$, so $C \leq N$ and hence $\bar{N} = G.$
\end{proof}

\begin{defn*}If $H$ is a subgroup of a group $G$, we define the series of successive normalizers $H_\alpha$ of $H$ by rules 
\begin{itemize}
\item[(i)] $H_0 = H,$
\item[(ii)] for any ordinal $\alpha$, $H_{\alpha + 1} = N_G(H_\alpha)$,
\item[(iii)] for limit ordinals $\mu$, $$H_\mu = \bigcup_{\alpha < \mu} H_\alpha.$$
\end{itemize}
\end{defn*}

\begin{defn*} Let $K$ be a subgroup of a group $G$, and $H$ a subgroup of $K$. Then we say that $H$ is an ascendant subgroup of $K$ if there exists a series of subgroups $(A_\alpha)_{\alpha \leq \rho},$ indexed by the ordinal numbers, satisfying the conditions \begin{itemize}
\item[(i)] $A_0 = H, A_\rho = K,$
\item[(ii)] for all $\alpha \leq \rho$, $A_\alpha \nsub A_{\alpha + 1}$
\item[(iii)] for limit ordinals $\mu,$ $$A_\mu = \bigcup_{\alpha < \mu} A_\alpha.$$
\end{itemize}
If this is the case, we write $H \text{ asc } K$, and note that if $\rho$ is finite, $H$ is subnormal in $K$.
\end{defn*}

\begin{lemma}\label{4.10}
Let $K$ be a subgroup of the locally nilpotent group $G$, and $H$ a subgroup of $K$. Let $H_\alpha$ denote the $\alpha^{th}$ normalizer of $H$, as in the above definition. Then 
\begin{itemize}
\item[(a)] $H = \bar{H}$ implies $H_\alpha = \bar{H}_\alpha$ for all ordinals $\alpha.$ 
\item[(b)] $H \text{ asc } K$ implies $\bar{H} \text{ asc } \bar{K}.$
\end{itemize}
\end{lemma}
\begin{proof}
(a) By induction on $\alpha.$ By hypothesis the result is true if $\alpha = 0.$ Assume that $H_\alpha = \bar{H}_\alpha,$ then by (\ref{4.9}(a)), 
$$
\bar{H}_{\alpha + 1} = \overline{N_G(H_\alpha)} \leq N_G(\bar{H}_\alpha) = N_G(H_\alpha) = H_{\alpha + 1}.
$$
If $\mu$ is a limit ordinal, assume $H_\beta = \bar{H}_\beta$ for all $\beta < \mu,$ and let $x \in \bar{H}_\mu.$ Then for some $\bar{\omega}$-number $n$, $x^n$ lies in $H_\mu$ and hence in $H_\beta$ for some $\beta < \mu$. $H_\beta$ is $\bar{\omega}$-isolated by assumption and so $x \in H_\beta$, a fortiori $x \in H_\mu$.

(b) By definition of $H \text{ asc } K,$ there exists a series $(A_\alpha)_{\alpha \leq \rho}$ satisfying the conditions set out above. Consider the series $(\bar{A}_{\alpha})_{\alpha \leq \rho}$. $\bar{A}_0 = \bar{H},$ $\bar{A}_\rho = \bar{K}.$ For all $\alpha$, $\bar{A}_{\alpha} \leq \bar{A}_{\alpha+1}$ and in addition, by (\ref{4.9}(a)), 
$$
N_G(\bar{A}_\alpha) \geq \overline{N_G(A_\alpha)} \geq \bar{A}_{\alpha + 1},
$$
which show that $\bar{A}_{\alpha} \nsub \bar{A}_{\alpha + 1}$. Let $\mu$ be a limit ordinal, then if $x$ is any element of $\bar{A}_\mu$, $x^n \in A_{\mu}$ for some $\bar{\omega}$-number $n$. By definition of $A_\mu,$ $x^n \in A_\alpha$ for some $\alpha < \mu,$ and so $x \in \bar{A}_\alpha.$ Hence $$\bar{A}_\mu = \bigcup_{\alpha < \mu} \bar{A}_\alpha$$ and this proves $\bar{H} \text{ asc } \bar{K}$, as required.
\end{proof}

\section{Basic Commutators}
If $G = \left<a_1, a_2, \ldots, a_q\right>$ is a finitely generated Abelian group, then every element of $G$ is expressible in the form $a_1^{r_1} a_2^{r_2} \ldots a_q^{r_q}.$ If $G$ is nilpotent, we know from (\ref{lemma1.7}) that a similar result holds, provided the generators $a_1, \ldots, a_q$ are supplemented by sufficiently many commutators. The object of the present section is to discover which commutators are really sufficient for this purpose. The main contributions to this problem are due to Magnus, Witt, and Marshall Hall.

\begin{defn*} Let $A$ be a set of elements of any group. For any element $b \in A$, let $A*b$ be the set of all elements $a= a_0, a_1, a_2, \ldots$ where $a_{k+1} = [a_k,b]$ and $a$ runs through all elements of $A$ other than $b$.
\end{defn*}

\begin{lemma}\label{5.1}
If $\left<A\right>$ is nilpotent, then $\left<A\right> = \left<b\right>\left<A*b\right>.$
\end{lemma}
\begin{proof}
Let $x \in \left<A\right>.$ Then $x$ is expressible as a product of elements of $A$ and their inverses. Among the factors in this expression, $b$ and $b^{-1}$ may occur, since $b \in A.$ We prove (\ref{5.1}) by transforming the expression for $x$ so that all the factors $b$ and $b^{-1}$ are moved step by step to the left where they will coalesce into a power of $b$, i.e  an element of $\left<b\right>.$ This is the process of \underline{collecting} $b$. A typical step takes us from $\ldots cb \ldots$ to $\ldots bc[c,b] \ldots \:$ . If $c \in A*b$, then so does $[c,b].$ However we must take into account the possible presence of inverses. Thus $\ldots c^{-1}b \ldots$ will be replaced by $\ldots b[c,b]^{-1}c^{-1} \ldots \: \:$. The nilpotence of $\left<A\right>$ is important only when we are moving $b^{-1}$. This involves changing $\ldots cb^{-1} \ldots$ into $\ldots b^{-1} c[c,b^{-1}] \ldots$ and $\ldots c^{-1}b^{-1} \ldots$ into $\ldots b^{-1}[c,b^{-1}]^{-1}c^{-1}\ldots \: .$ If we write $c= c_0$ and define $c_{k+1} = [c_k, b],$ then $c_k$ will belong to $A*b$ if $c$ does. But $c_k \in \gamma_{k+1}(\left<A\right>)$ so that $c_k = 1$ whenever $k$ is not less than the class of $\left<A\right>.$ Let $c_k^* = [c_k, b^{-1}].$ Then $1= [c_k, bb^{-1}] = c_k^* c_{k+1}b^{-1}$ by (\ref{lemma2.1}) so that $c_k^* = (c_{k+1}^*)^{-1}c_{k+1}^{-1}$. Using this formula repeatedly, and remembering that $c_k^*$ is also $1$ for large $k$, we obtain a finite expression for 
$$
c_0^* = [c, b^{-1}] = c_2 c_4 c_6 \ldots \quad \quad \ldots c_5^{-1}c_3^{-1}c_1^{-1}
$$ 
in terms of elements of $A*b.$ Thus, after a finite number of steps, the process ends with an expression for $x$ of the form $b^rx'$, where $r$ is an integer and $x' \in \left<A*b\right>.$
\end{proof}

The next part of the argument is purely formal. For the sake of simplicity let us suppose we are dealing with a finite number of symbols forming the set $X = x_1, x_2, \ldots, x_q$ which can be combined into more complicated expressions by a binary operation written $[u,v].$ We assume $q>1.$

\begin{defn*} $C = C(X)$ is the smallest set of expressions such that
\begin{itemize}
\item[(i)] $x_i \in C$ for each $i=1,2, \ldots, q;$ and
\item[(ii)] if $u$ and $v$ belong to $C$, so does $[u,v].$
\end{itemize}
\noindent The \underline{weight} of an element $u$ of $C$ is the total number of occurrences of the variables $x_i$ in $u$. Thus $\text{wt}([u,v]) = \text{wt}(u) + \text{wt}(v).$
\end{defn*}

\noindent A \underline{basic sequence} in the variables $x_1, \ldots, x_q$ is any sequence
\begin{equation}\label{(1)section4}
    b_1, b_2, b_3, \ldots \tag{1}
\end{equation}
of elements of $C$ which can be generated according to the following rules:
\begin{itemize}
\item[(i)] $b_1 \in X = X_0$ and, for $k>0,$ $b_k \in X_{k-1}$ and $X_k = X_{k-1}*b;$
\item[(ii)] if $i < j$, then $\text{wt}(b_i) \leq \text{wt}(b_j)$.
\end{itemize}
Thus, if $b_1, \ldots, b_{k-1}$ have already been chosen so as to satisfy \textit{(i)} and \textit{(ii)}, then $X_{k-1}$ is a well-defined subset of $C$ and $b_k$ can be chosen to be an arbitrary element of least possible weight in $X_{k-1}.$ Without loss of generality, we may suppose that the first $q$ terms in \textit{(i)} are $x_1, \ldots, x_q$. Next there will come the $\frac{1}{2}q(q-1)$ terms $[x_i, x_j]$ with $i > j$ in some order. Then a batch of terms of weight $3$ and so on.

\begin{defn*} Let $P_i \: (i \geq 0)$ denote the set of all formal products $p_i$ of the form $$p_i = b_1^{r_1}b_2^{r_2} \ldots b_i^{r_i}q_i,$$ where each $r_i$ is an integer $r_i \geq 0$ and $q_i$ is any product (possibly empty) of elements of $X_i$. We define the weight of such a product $p_i$ to be the sum of the weights of its factors.
\end{defn*}

\begin{lemma}\label{lemma5.2}
	For each $i > 0$, there is a $1-1$ weight-preserving correspondence $\theta_i$ which maps $P_{i-1}$ onto $P_i$.
\end{lemma}
\begin{proof} For, let $$p_{i-1} = b_1^{r_1} \ldots b_{i-1}^{r_{i-1}}q_{i-1}$$ be any element of $P_{i-1}$, so that $q_{i-1}$ is a product of elements of $X_{i-1}.$ Hence $q_{i-1}$ has the form
\begin{equation}\label{(2)section5}
    q_{i-1} = b_i^{r_i} \ldots c_1 b_i^{s_1} \ldots c_2 b_i^{s_2} \ldots \quad \quad \ldots c_n b_i^{s_n} \ldots, \tag{2}
\end{equation}
where $r_i \geq 0,$ $n \geq 0$ and $s_1, s_2, \ldots, s_n$ are positive integers, while $c_1, c_2, \ldots, c_n$ are elements $\neq b_i$ from $X_{i-1}$ (so the $c_j$ all belong to $X_i)$ and the $\ldots$ stand for possibly empty products of such elements. Define $p_{i-1}\theta_i= p_i$ to be the element of $P_i$ which is obtained by replacing each of the $n$ products $c_a b_i^{s_a}$ by the "corresponding commutator" $[\ldots[[c_a, b_i], b_i] \ldots, b_i]$ with $s_a$ terms $b_i.$ Denote this commutator by $c_a^{(s_a)}$ for short.

Conversely, if $p_i = b_1^{r_1} \ldots b_i^{r_i} q_i$ is any element of $P_i,$ then $q_i$ is a product of elements of $X_i$ and therefore has the form 
$$
\ldots c_1^{(s_1)} \ldots c_2^{(s_2)} \ldots \quad \ldots c_n^{(s_n)} \ldots,
$$
where the $s_a$ are positive integers, $n \geq 0,$ and $c_1, c_2, \ldots, c_n$ are elements of $X_{i-1}$ different from $b_i$, while the $\ldots$ stand for possible empty such products. Define $p_i \theta_i^{'} = p_{i-1}$ to be that element of $P_{i-1}$ which is obtained by replacing each $c_a^{(s_a)}$ by the corresponding product $c_a b_i^{s_a}.$ Clearly $\theta_i$ and $\theta_i'$ are mutually inverse and so $\theta_i$ maps $P_{i-1}$ onto $P_i$ and is $1-1.$ It is equally clear that the $\theta_i$ are weight preserving.
\end{proof}

\begin{defn*}	A basic product is any finite product of the form $b_1^{r_1}b_2^{r_2} \ldots b_N^{r_N}$ for some $N$, where the $r_i$ are integers $\geq 0$. Since the basic sequence contains only a finite number of terms of given weight $w$, there exists for each $w$ a number $N_w$, such that all elements of $P_i$ of weight $w$ are basic products, provided $i \geq N_w$.
\end{defn*}

\begin{cor*}
	The number of distinct basic products of weight $w$ in the variable $x_1,x_2, \ldots, x_q$ is precisely $q^w$.
	\end{cor*}
\begin{proof}
	For the elements of weight $w$ in $P_0$ are just products $x_{i_1} x_{i_2} \ldots x_{i_w}$ and these are mapped $1-1$ by $\theta_1, \theta_2, \ldots, \theta_i$ onto the set of all basic products of weight $w$, provided $i \geq N_{w}$.
\end{proof}

Consider now the set $R_w$ of all linear combinations with integer coefficients of the $q^w$ products $x_{i_1} x_{i_2}\ldots x_{i_w};$ and interpret $R_0$ to mean the ring of rational integers. With the obvious (distributive and associative) multiplication, the direct sum $R$ of the additive groups  $R_0, R_1, \ldots$ becomes the \underline{free associative ring} generated by $x_1, \ldots, x_q$. Define $[u,v]$ for $u$ and $v$ in $R$ to mean $uv - vu.$ The elements of $C = C(X)$ are then called the \underline{Lie elements} of $R$. The $q^w$ basic products of weight $w$ are then all in $R_w$ and we have
\begin{lemma}\label{5.3}
	The basic products of weight $w$ form a basis for the free additive abelian group $R_w$.
\end{lemma}
\begin{proof}
Since the number of such basic products is equal to the rank of $R_w$, it is enough to show that, for $i>0$, every $p_{i-1} \in P_{i-1}$ is expressible as an integral linear combination of elements of the same weight in $P_i.$ For then, every product $x_{i_1} x_{i_2} \ldots x_{i_w} \in P_0$ will be expressible as an integral linear combination of elements of weight $w$ in $P_i,$ for each value of $i$; and these elements are basic products for $i \geq N_w.$ Take $p_{i-1}$ as in proof of (\ref{lemma5.2}) with $q_{i-2}$ given by (\ref{(2)section5}). For any $U,v$ in $R$, we have
\begin{equation}\label{(3)section5}
[u, \underbrace{v,v, \ldots, v}_s] = uv^s - svuv^{s-1} + {s \choose 2} v^2 uv^{s-2} - \ldots.\tag{3}
\end{equation}

If $u \in X_i$ for $v = b_i,$ then $[u, v, v, \ldots, v]$ also belongs to $X_i$. We may use (\ref{(3)section5}) to replace the products $\ldots c_a b_i^{s_a} \ldots$ which occur in $q_{i-1}$ by integral linear combinations of other products in which the number of pairs of factors $\ldots c \ldots b_i \ldots$ (where some $c$ in $X_i$ precedes some factor $b_i$) is smaller than before. In this way, after a finite number of steps, $q_{i-1}$ will be expressed as a linear integral combination of terms in which no factor $c \in X_i$ precedes any $b_i$. These terms all have the form $b_i^{r_i}$ product of elements of $X_i,$ and the result follows.
\end{proof}

\begin{cor*} The terms of the basic sequence $b_1,b_2,\ldots$ are linearly independent.\end{cor*}

We now relate the interpretations $uv - vu$ and $u^{-1}v^{-1}uv$ of the operation $[u,v]$. It is convenient for this purpose to select an arbitrary positive integer $n$ and to pass from $R$ to the ring $U= R / \Fr{k}^n$, where $\Fr{k}$ is the ideal $R_1 + R_2 + \ldots$ of $R$. By  (\ref{theorem3.6}), the set $1 + \Fr{k} / \Fr{k}^n$ is then a nilpotent group of class $<n$. Let $H$ be the subgroup generated by the $q$ elements $y_1,y_2,\ldots,y_q$, where $y_i = 1+ x_i$. Corresponding to the basic sequence $b_1,b_2,\ldots$ in the variables $x_i$, formed using the operation $uv - vu$, we shall have a parallel basic sequence $c_1,c_2,\ldots$ in the variables $y_i$ formed with the operation $u^{-1}v^{-1}uv$. Each $c_j$ is obtained from the corresponding $b_j$ by replacing every $x_i$ by $y_i$ and reinterpretation. Suppose that $b_N$ is the last term of weight $< n$. Then $b_{N+1},b_{N+2},\ldots$ all belong to $\Fr{k}^n$ and hence vanish in $U$; while $c_{N+1},c_{N+2},\ldots$ all belong to $\gamma_n(H) \leq 1 + \Fr{k}^n / \Fr{k}^n = 1$.

\begin{lemma}\label{5.4}
	For $1 \leq j < n$, we have 
    $$
    c_j \equiv 1 + b_j \text{ mod } \Fr{k}^{w+1}
    $$ 
    where $w$ is the weight of $b_j$ (and $c_j$).
\end{lemma}

\begin{proof}
	This is clear for $w =1$ since then $c_j = y_j = 1+x_j = 1+b_j$. For $w > 1$, we have $b_j = [b_k,b_\ell]$ and $c_j = [c_k,c_\ell]$, where the weights $r$ and $s$ of $b_k$ and $b_\ell$ add up to $w$. We may assume that 
    $$
    c_k \equiv 1 + b_k \text{ mod } \mathfrak{k}^{r+1} \quad \text{ and } \quad c_\ell \equiv 1 + b \text{ mod } \mathfrak{k}^{s+1}.$$ If $c_k = 1 + \beta_k$, $c_\ell = 1 + \beta_\ell$, we  have
	$$
	c_k c_\ell = 1 + \beta_k + \beta_\ell + \beta_k \beta_\ell = \pr{1 + \beta_k + \beta_\ell + \beta_\ell \beta_k} \pr{1 + \beta_k \beta_\ell - \beta_\ell \beta_k +} \:	\text{ terms in } \: \mathfrak{k}^{r+s + 1}),$$ so that $$[c_k,c_\ell] = 1 + \beta_k \beta_\ell - \beta_\ell \beta_k \equiv 1 +  [b_k,b_\ell] \text{ mod } \mathfrak{k}^{r+s+1}
	$$
	since 
	$$
	\beta_k \equiv b_k \text{ mod } \mathfrak{k}^{r+1} \quad \text{ and } \quad \beta_\ell \equiv b_\ell \text{ mod } \mathfrak{k}^{s+1}.
	$$\end{proof}

\begin{lemma}\label{5.5}
	The elements of $H = \left<y_1,\ldots,y_q\right>$ are uniquely expressible in the form $$c_1^{\mu_1} c_2^{\mu_2} \ldots c_n^{\mu_n}$$ with integers $\mu_1,\ldots,\mu_n$.
\end{lemma}
\begin{proof}
	Since $H$ is nilpotent we may apply Lemma \ref{5.1} repeatedly to obtain $$H = \left<c_1\right> \left<c_2\right> \ldots \left<c_n\right>  \left<Y_n\right>$$ where $Y = y_1,\ldots,y_q$ and the $Y_i$ are defined analogously to the $X_i$. But $Y_n$ consists exclusively of commutators of weight $\geq n$. Thus $Y_n \leq \gamma_n(H) = 1$. Thus every element $u$ of $H$ is expressible in the form $c^\mu$ (to use a vectorial notation). Suppose if possible that $u = c^\lambda$ also, with $\lambda \neq \mu$. Let the least suffix where $\lambda$ differs from $\mu$ be $i$ and let $c_i$ have weight $w$. Suppose that $c_{i+1},\ldots,c_{i+\ell}$ are the later terms which also have weight $w$. We then obtain 
	$$
	v = c_i^{\lambda_i} c_{i+1}^{\lambda_{i+1}} \ldots = c_i^{\mu_i} c_{i+1}^{\mu_{i+1}}\ldots
	$$
	with $\lambda_i \neq \mu_i$. By (\ref{5.4}), $v$ is congruent $\text{ mod } \mathfrak{k}^{w+1}$ to both $$\sum_{a=0}^\ell \lambda_{i+a} b_{i+a} \quad \text{and} \quad \sum_{a=0}^\ell \mu_{i+a} b_{i+a}.$$ But this is impossible by  (\ref{5.3}), since $b_i,b_{i+1},\ldots,b_{i+\ell}$ are linearly independent elements of $v_w$.
\end{proof}

\begin{thm}\label{5.6}
	Let $F$ be the free group on $\bar{y}_1,\ldots,\bar{y}_q$. Then $H \equiv F / \ga_n(F)$.
\end{thm}
\begin{proof}
	For there is a unique homomorphism $\theta$ which maps $F$ onto $H$ and each $\bar{y}_i$ to $y_i$. Let $K$ be the kernel of $\theta$. Then $F / K \cong H$. But $H$ is of nilpotent class $< n$. Hence $\gamma_n(F) \leq K$. Apply (\ref{5.1}) to the group $G = F / \gamma_n(F)$ which is nilpotent of class $< n$. In the obvious notation, we obtain $F = \left<\bar{c}_1\right> \left<\bar{c}_2\right> \ldots \left<\bar{c}_n\right>\gamma_n(F)$, since $\left<\bar{Y}_n\right> \leq \gamma_n(F).$ Therefore, if $K > \gamma_n(F),$ there would exist integers $\mu_1, \ldots, \mu_n$ not all zero such that $c_1^{\mu_1} c_2^{\mu_2} \ldots c_n^{\mu_n} = 1$. This is contrary to (\ref{5.5}).
\end{proof}

Theorem \ref{5.6} may be supplemented by two further remarks. Since $n$ was arbitrary, we have at once the following.

\begin{cor*}
	Let $F$ be freely generated by $\bar{y}_1,\ldots,\bar{y}_q$ and let $\bar{c}_1,\bar{c}_2,\ldots$ be any basic sequence in $\bar{y}_i$'s. For every $w > 0$, $\gamma_w(F) / \gamma_{w+1}(F)$ is a free abelian group with a basis represented by those $\bar{c}_j$ which have weight $w$.
	\end{cor*}

\begin{thm}[Witt]\label{5.7}
	The rank of $\gamma_w(F) / \gamma_{w+1}(F)$ is $n(w,q)$ where
	$$
	n(w,q) = \frac{1}{w}  \sum \mu(d) q^{\frac{w}{d}}.
	$$
	Here $\mu(d)$ is the M\"{o}bius function, $= (-1)^r$ if $d$ is a product of $r \: (\geq0)$ distinct primes and otherwise zero.
\end{thm}
\begin{proof}
	 For by (\ref{lemma5.2}) corollary, we have the power series identity
	$$
	\frac{1}{1 - qt} = \prod_{w=1}^{\infty}\frac{1}{(1-t^w)^{n(w,q)}}.
	$$
	Taking logs, equating coefficients of powers of $t$ using the M\"{o}bius inversion formula gives (\ref{5.7}).
\end{proof}

\begin{defn*} Let $\mathcal{P}$ be any property of groups. We call a group $G$ \underline{residually}-$\mathcal{P}$ if, for every $x \neq 1$ in $G$, there exists a subgroup $K$ such that $K \nsub G$, $K/G$ is a $\mathcal{P}$-group, and $x \notin K$.
\end{defn*}

For example, a group $G$ is \underline{residually-finite-$p$}, where $p$ is a prime, if the intersection of the normal subgroups of $G$ of index a power of $p$ is the unit subgroup. Since finite $p$-groups are nilpotent, it follows that residually finite-$p$ groups are residually nilpotent. A group $G$ is residually nilpotent if and only if $$\bigcap_{n=1}^{\infty}\gamma_n(G) = 1.$$ For if $K \nsub G$ and $G/K$ is nilpotent, then $\gamma_n(G) \leq K$ for some $n$.

\begin{thm}\label{5.8}
	Every free group $F$ is residually finite-$p$ for all primes $p$. In particular $$\bigcap_{n=1}^{\infty}\gamma_n(F) = 1.$$
\end{thm}
\begin{proof}
	For let $F$ be freely generated by $x_1,x_2,\ldots$ and let $y = x_{i_1}^{r_1} x_{i_2}^{r_2} \ldots x_{i_n}^{r_n}$ be any element $\neq 1$ in $F$. We may suppose that $n > 0$, $r_1,r_2,\ldots,r_n \neq 0$, and $i_a \neq i_{a-1}$ for $a = 2,3,\ldots,n$. Given the prime $p$, let $r_i = p^{s_i}m_i$, where $m_i$ is prime to $p$ (and possibly negative). Let $N = p^{s_1} + p^{s_2} \ldots + p^{s_n}$ and suppose without loss generality that $i_1,i_2,\ldots,i_n$ are positive integers $\leq q$. Let $\mathbb{K}$ be the field of $p$ elements and let $V$ be the vector space of dimension $1 + q + q^2 + \ldots + q^N$ over $\mathbb{K}$ having the basis given by the set of all formal products $u_{j_1}u_{j_2} \ldots u_{j_r}$ $(0 \leq r \leq N)$ of the $q$ non-commuting variables $u_1,\ldots,u_q$. Make $V$ an associative ring by defining a distributive multiplication in $V$ according to the rule
	\begin{align*}
    (u_{j_1} \ldots u_{j_r}) \cdot (u_{k_1} \ldots u_{k_s}) &= u_{j_1} \ldots u_{j_r} u_{k_1} \ldots u_{k_2} & \quad \: r+s \leq N\\
     &= 0 & \quad \text{ otherwise.}
	\end{align*}

	Then the elements of $V$ with "constant term" $1$ form a group $G$, the inverse of $1+v \in G$ being $1 - v + v^2 - \ldots \pm v^n$, $G$ is of order $p^m$ where $m = q + q^2 + \ldots + q^N$. Let $\theta$ be the homomorphism of $F$ into $G$ which maps $x_i$ onto $1+u_i$ for $1 \leq i \leq q$ and the remaining $x_j$ onto $1$. If $K$ is the kernel of $\theta$, then $F/K$ is a finite $p$-group since the image of $F$ under $\theta$ is a subgroup of $G$. But $$y' = \prod_{u=1}^n(1+u_{i_a})^{r_a}$$ and the coefficient $u_{i_1}^{p^{s_1}}\ldots u_{i_n}^{p^{s_n}}$ in $u^\theta$ is $\prod_a {p^{s_i } \: m_i \choose p^{s_i}}$, which is $\neq 0$ since all all these binomial coefficients are prime to $p$. Thus $y \notin K$.
\end{proof}

K. Gruenberg has shown among other similar results that every torsion-free and finitely generated nilpotent group is residually finite-$p$ for all primes $p$.

\section{Embedding Theorems}
Let $p = y_1y_2 \ldots y_n$ be any product of elements of a group. Suppose the $n$ factors of $p$ to be divided arbitrarily into $r$ mutually exclusive subsets $X_1,X_2,\ldots,X_r$, for example by labeling each $y_i$ with one of the labels $(1),(2),\ldots,(r)$. Denote by $X_{\alpha \beta}$, where $1 \leq \al < \beta \leq r$, the set of all commutators of the $y_i$'s whose components are drawn exclusively from the sets $X_\alpha$ and $X_\beta$, and have at least one component from each of these sets. More generally, let $S$ be any non-empty subset of the set $R = (1),(2), \ldots, (r)$ of labels; and denote by $X_S$ the set of all commutators of the $y's$ whose components are drawn exclusively from the sets $X_a$ with $(a) \in S$ and which have at least one component in each of these sets. Order the non-empty subsets $S$ of $R$ in order of increasing magnitude $|S|$ and, for given $|S|$, lexicographically.

\begin{lemma}\label{6.1}
	$p = \prod_{S \leq R} q_S$, where $q_S$ is a product of elements of $X_S$ and the $2^r - 1$ factors occur in the given ordering of the subsets $S$.
\end{lemma}
\begin{proof}
	This is proved by transforming the product $p=y_1 y_2 \ldots y_n$ step by step to the required form. First we \underline{collect} those $y_i$ which belong to $X_1$ by moving them step by step to the left, to yield the factor $q_1$ in (\ref{6.1}). A typical move is from $\ldots uv \ldots$ to $\ldots vu[u,v] \ldots$ where $v$, the term being collected, belongs to $X_1$ and $u$ does not. Therefore $[u,v]$ does not belong to $X_1.$ If $u \in X_S$, then $[u,v]$ will belong to $X_T$ where $T$ is the union of $S$ and $(1)$. After all terms in $X_1$ have been collected in this way, we obtain $p = q_1p',$ where $p'$ is a product of factors from the set $X_2, \ldots, X_r, X_{12}, \ldots, X_{1r}$. We then collect the factors in $p'$ which belong to $x_2$ and so on.
\end{proof}

Let $p_S$ be the product $y_{i_1} y_{i_2} \ldots y_{i_m}$ where $i_1 < i_2 < \ldots < i_m$ and the $y_{i_a}$ are precisely those $y's$ which belong to some $X_a$ with $(a) \in S$. Thus $p_S$ is obtained from $p$ by equating to $1$ all the remaining $y's$. But a commutator which has a component equal to $1$ is itself equal to $1$. Hence
\begin{lemma}\label{6.2}
	$p_S = \prod_{T \leq S} q_T$, where the factors $q_T$ are ordered in the given ordering of subsets.
\end{lemma}

(\ref{6.2}) demonstrates that the $2^r - 1$ elements $q_S$ are expressible in terms of the $2^r - 1$ elements $p_s$ by recurrence. For example, $q_a = p_a$ and, for $a < \beta$, $q_{a_\beta} = p_{\beta}^{-1} p_a^{-1} p_{a_\beta}$ and so on.

Now take $p = x_1^r x_2^r \ldots x_n^r$ with $r>0$, and label the factors according the scheme
$$
x_1^{(1)} x_1^{(2)} \ldots x_1^{(r)} x_2^{(1)} \ldots x_2^{(r)} \ldots x_n^{(1)} \ldots x_n^{(r)}.
$$
If $|S| = w$, we have $p_S = x_1^w x_2^w \ldots x_n^w$ which depends only on $|S|$. Hence $q_S$ also depends only on $|S|$ and when $|S| = w$, we may write $q_s = \tau_w(x_1,\ldots,x_n) = \tau_W(x)$. By (\ref{6.2}), the words $\tau_w$ are defined by the recurrence relations
$$
x_1^w x_2^w \ldots x_n^w = \tau_1(x)^w \tau_2(x)^{w \choose 2} \ldots \tau_k(x)^{w \choose k} \ldots \tau_w(x).
$$
We call these \underline{Petresco words.} Note that $\tau_1(x) = x_1 x_2 \ldots x_n$.

\begin{thm}\label{6.3}
	For any group $G$, $\tau(G) \leq \gamma_w(G)$. 
\end{thm}
\begin{proof}
	For $\tau_{w}(x) = q_s$ with $|S| = w$ and $q_s$ is a product of commutators each of which has at least $w$ components.
\end{proof}

\begin{thm}\label{6.4}
	Let $G$ be a finite group which is not nilpotent. Then there exists elements $x_1,x_2$ in $G$ such that $\tau_{n}(x_1,x_2) \neq 1$ for every $n=1,2,\ldots \: .$
\end{thm}
\begin{cor*} $G$ finite and $\tau_n(G) = 1$ for some $n$ implies $G$ nilpotent.\end{cor*}

Now let $G$ be a torsion-free, finitely generated nilpotent group. By (\ref{4.7}), each of the groups $G/Z_i$ is also torsion-free. By (\ref{lemma1.9}), $Z_{i+1}$ is finitely generated. Hence each $Z_{i+1} / Z_i$ is a free Abelian group of finite rank. Thus we can refine the upper central series of $G$ to obtain a central series
$$
G = G_0 > G_1 > \ldots > G_n = 1
$$
for which each $G_{i-1} / G_i$ is an infinite cyclic group. Choose $u_i$ so that $G_{i-1} = \left<G_i,u_i\right>$ for each $i=1,2, \ldots, n$. Then every element $x \in G$ is uniquely expressible in the form
$$
x = u_1^{\xi_1}u_2^{\xi_2} \ldots u_n^{\xi_n} = u^\xi \quad \text{ for short, }
$$
with integers $\xi_1,\ldots, \xi_n$. Call $u = \pr{u_1,\ldots,u_n}$ as defined in this way a \underline{canonical basis} of $G$ and $\xi = \pr{\xi_1,\ldots,\xi_n}$ the \underline{canonical parameters} of $x$ with respect to the basis $u$.

Let $y = u^\eta$ be any other element of $G$ and let $xy = u^\zeta$. Then each $\zeta_i$ is a function of the $2n$ integer variables $\xi_i,\eta_k$ if $x$ and $y$ are allowed to vary in $G$. Again, let $\lambda$ be any integer and let $x^\lambda = u^{\bar{\omega}} = u_1^{\bar{\omega}_1} \ldots u_n^{\bar{\omega}_n}$. Then each $\bar{\omega}_i$ is a function of the $n+1$ integer variables $\xi_i$ and $\lambda$.

\begin{thm}\label{6.5}
	The functions $\zeta_i$ and $\bar{\omega}_i$ ($i=1,2, \ldots, n)$ are polynomials.
\end{thm}

That means, strictly speaking, that there exist polynomials $f_i(\xi,\eta)$ of the $2n$ arguments $\xi_j$,$\eta_k$ such that, when these arguments are given integer values, $\zeta_i = f_i(\xi,\eta)$. Similarly, for $\bar{\omega}_i$. But a polynomial is uniquely determined by its values for integer values of the arguments, so that no ambiguity can arise if we identify the function $\zeta_i$ with the polynomial $f_i$. Clearly, $\zeta_i$ and $\bar{\omega}_i$ must be integer-valued polynomials. Thus $\zeta_i$, for example, can be expressed as a linear combination with integer coefficients of binomial products of the form
$$
{\xi_i \choose r_1}{\xi_2 \choose r_2} \ldots {\xi_i \choose r_i}{\eta_1 \choose s_1} \ldots {\eta_i \choose s_i}
$$
where $r_j$ and $s_k$ are integers $\geq 0$. For $G_i \nsub G$ and so $\zeta_i$ can only depend on $\xi_1,\ldots,\xi_i,\eta_1,\ldots,\eta_{i}$. Similar remarks apply to $\bar{\omega}_i$.

\begin{proof}[Proof of Theorem \ref{6.5}]
If $n = 1$, we have $\zeta_1 = \xi_1 + \eta_1$ and $\bar{\omega}_1 = \lambda \xi_1$. Let us call the part of (\ref{6.5}) which concerns the $\zeta_i$ the statement $(M_n)$ and the part which concerns the $\bar{w}_i$ the statement $(E_n)$. Suppose that $n>1$, and assume that $(M_i)$ and $(E_i)$ have been proved for all $i < n$. We first deduce $(M_n)$.
\begin{equation}\label{(1)section6}
x y = u^{\xi} \cdot u^\eta = u_1^{\eta_1 + \eta_2} \prod_{i=2}^{n}\pr{u_1^{-\eta_1} u_i^{-1} u_1^{\eta_1}}^{\xi_i} u_2^{\eta_2} \ldots u_n^{\eta_n}. \tag{1}
\end{equation}
But $$u_1^{-\eta_1} u_i^{-1} u_1^{\eta_1} = u_1^{-\eta_1} \pr{u_i^{-1} u_1 u_i}^{\eta_1} u_i^{-1}$$ and $$u_i^{-1} u_1 u_i = u_1 u_{i+1}^{c_{i,1}} \ldots u_n^{c_{i,n-i}}$$ with certain constants $c_{ij}$. The group $H_i = \left<u_1,u_{i+1},\ldots,u_n\right>$ has the canonical basis $\pr{u_1,u_{i+1},\ldots,u_n}$ with $n-i+1$ terms and so, if $i > 1$, we have $$\pr{u_i^{-1} u_1 u_i}^{\eta_1} = u_1^{\eta_1} u_{i+1}^{\phi_{i,1}} \ldots u_n^{\phi_{i,n-1}}$$ where by $(E_{n-i+1})$, the $\phi_{ij}$ are polynomials in $\eta_1$. Then $(M_{n-i+1})$ gives $$u_1^{-\eta_1} u_i^{-1} u_1^{\eta_1} = u_{i+1}^{\phi_{i,1}} \ldots u_n^{\phi_{i,n-1}} \cdot u_i^{-1} = u_{i}^{-1} u_{i+1}^{\psi_{i,1}} \ldots u_n^{\psi_{i,n-1}}$$ where the $\psi_{ij}$ are again polynomials in $\eta_1$. Hence 
$$
\pr{u_1^{-\eta_1}u_i^{-1}u_1^{\eta_1}}^{-\xi_i} = u_i^{\eta_i} u_{i+1}^{\theta_{i,1}} \ldots u_n^{\theta_{i,n-1}}
$$
where by $(E_{n-i+1})$ again the $\theta_{ij}$ are polynomials in $\eta_1$ and $\xi_i$. Substituting in (\ref{(1)section6}) and using $(M_{n-1})$ $n-1$ times, we obtain $(M_n)$.

Let $(E_n^+)$ mean $(E_n)$ for $\lambda > 0$. Let $v_i = \tau_i(u_1^{\eta_1}, \ldots, u_n^{\eta_n})$. Each $\tau_i$ is a certain word in its arguments. Hence repeated applications of $(M_n)$ give $v_i = u_{1}^{a_{i1}} \ldots u_n^{a_{in}}$ where $a_{ij}$ are polynomials in the $\eta$'s. Also if $c$ is the class of $G$, we have $v_{c+1} = v_{c+2} = \ldots = 1$ by (\ref{6.3}) since $\tau_{c+1}(G) = 1$. But $v_1 = u_1^{\eta_1} u_2^{\eta_2} \ldots u_n^{\eta_n} = x$ and so, from the definition of the $\tau$'s for $\lambda > 0$,
\begin{equation}\label{(2)section6}
x^\lambda = u_1^{\lambda \eta_1} u_2^{\lambda \eta_2} \cdot u_n^{\lambda  \eta_n} v_c^{- {\lambda \choose c}} v_{c-1}^{-{\lambda \choose c - 1}} \ldots v_2^{-{\lambda \choose 2}}. \tag{2}
\end{equation}
But $v_2, v_3, \ldots, v_c$ all lie in $G^\prime = \tau_2(G)$ by (\ref{6.3}) and $G^\prime \leq G_1$. So, by $(E_{n-1})$, each $v_i^{-{\lambda \choose i}} = u_2^{b_{i2}} \ldots u_n^{\beta_{in}}$, where the $b_{ij}$ are polynomials in $\lambda$ and the $\eta$'s. Substituting (\ref{(2)section6}) and applying $(M_{n-1})$ several times now yields $(E_n^+)$, with uniquely determined polynomials $\bar{\omega}_i = \bar{\omega}_{i}(\lambda,\eta)$.

By $(M_n)$, $x^{-1} = u_n^{-\eta_n} \ldots u_1^{-\eta_1} = u^\delta$, where the $\delta_i$ are polynomials in the $\zeta$'s. By $(E_n^+)$, we have, if $\lambda > 0$, $x^{-\la} = u^{\eps}$ where $\eps_i$ are polynomials in the $\eta$'s and $\lambda$. Hence, if $\lambda > 0$ and $\mu > 0$, $(M_n)$ and $(E_n^+)$ give $x^{\mu-\la} = u^k$ where $k_i = k_k(\la,\mu,\eta)$ is a polynomial in $\lambda$, $\mu$, and the $\eta$'s. Therefore, if $\lambda > 0$, and $\mu > 0$, we have that $\bar{\omega}_i(\mu,\eta) = k_i(\la, \lambda + \mu, \eta)$. This must therefore be a polynomial identity. Given any integer $\mu$ choose $\lambda$ so that $\lambda + \mu$ and $\lambda$ are both positive. Then 
$$
x^\mu= x^{\lambda + \mu - \lambda} = u_1^{k_1(\lambda,\lambda + \mu,\eta)} \ldots u_n^{k_n(\lambda, \lambda + \mu,\eta)} = u_1^{\bar{\omega}_1(\mu,\eta)} \ldots u_n^{\bar{\omega}_n(\nu,\eta)}.
$$
and $(E_n)$ is proved.
\end{proof}

Theorem \ref{6.5} allows us to embed $G$ in various larger groups with more or less interesting properties.

Let $\nu$ be any field of characteristic $0$, or, more generally, any domain of integrity which contains the rational integers and is such that for $\lambda \in \nu$ implies that
$$
{\lambda \choose n} = \frac{\lambda (\lambda-1) \ldots (\lambda - n+1)}{n!} \in \nu
$$
for all $n = 1,2 \ldots \: .$ For example, $\nu$ could be the ring of $p$-adic integers where $p$ is a rational prime. Let such a ring be called a binomimal ring.

Define $G^\nu$ to consist of all form products $u^\eta$ with exponents $\eta_1, \ldots, \eta_n$ in $\nu$ and let multiplication be defined in $G^\nu$ and also exponentiation by an arbitrary $\lambda$ of $\nu$ by means of the polynomials $\zeta_i(\eta,\mu)$ and $\bar{\omega}(\la,\eta)$ in the obvious way. If $\nu$ is a binonomial ring, the values of these polynomials for arguments in $\nu$ themselves belong to $\nu$. With this definition of multiplication $G^\nu$ becomes a group, for the associative law and the other group axioms which hold in $G$ express themselves by certain polynomial identities which must be satisfied by the $\zeta_i$ and $\bar{\omega}_i$. For the same reason, if we define $G_i^\nu$ to consist of all $u^\eta$ for which $\zeta_1 = \zeta_2 = \ldots = \zeta_i = 0$, then $$G^\nu= G_0^\nu > G_1^\nu > \ldots > G_n^\nu = 1$$ is a central series of $G^\nu$ and each $G_{i-1}^{\nu} / G_{i}^\nu$ is isomorphic to the additive group of $\nu$. Again, $G^\nu$ is nilpotent of the same class as $G$.

The most interesting cases are as follows.

(i) Let $\nu$ be the field of rationals. Then $x \in G^\nu$ has some positive integral power $x^\lambda$ in $G$. Also, given $x$ and an integer $m > 0$, there exist in $G^\nu$ a uniquely determined element $y$ such that $y^m = x$. We express this by saying that $G^\nu$ is \underline{radicable}. In this case, it can be shown that every automorphism of $G$ can be extended in one and only one way to an automorphism of $G^\nu$.

(ii) Let $\nu$ be the field of real numbers. Then $G^\nu$ is a Lie group of dimension $n$, which is topologically an $n$-dimensional Euclidean space.

(iii) Let $\nu$ be the ring of $p$-adic integers, where $p$ is a rational prime. In this case, $G^\nu$ is the \underline{$p$-adic completion} of $G$, that is, the completion of $G$ is obtained by taking as neighborhoods of the unit element the set of all normal subgroups $K$ of $G$ such that $|G:K|$ is a power of $p$. It has been shown by Gruenberg that, if $G$ is a torsion free, finitely generated nilpotent group, then the common intersection of all these subgroups $K$ is $1$.

These embedding theorems for torsion free, finitely generated nilpotent groups were first obtained by Malcev, using the theory of Lie rings. That any such group can be embedded in a Lie group of the "correct" dimension has also been proved by Jennings, using a different method, but also appealing to the theory of Lie algebras. By standard procedures, the result for case $(i)$, where $\nu$ is the field of rationals , can be extended to arbitrary torsion free, locally nilpotent groups $G$: every such $G$ can be embedded in a radicable torsion-free locally nilpotent group $\bar{G} = G^\nu$ such that if $\bar{\omega}$ is the set of all primes, then $\bar{G} = G_{\bar{\omega}}$. This group $\bar{G}$, the \underline{absolute isolator of $G$}, is determined by $G$ to within isomorphism. More precisely, every automorphism of $G$ extends to a uniquely determined automorphism of $\bar{G}$. This result was first obtained by Malcev.

Since the groups $G^\nu$ admit exponentiation by arbitrary elements of the binomial ring $\nu$, they are somewhat analogous to $\nu$-modules, though with multiplicative instead of the usual additive notation. The principal differences are two: First, $G$ may be any suitable locally nilpotent group; it need not be Abelian, as it is in the module case. Secondly, $\nu$ is restricted to be a binomial ring in general, though this restriction may be weakened considerably, especially when $G$ is nilpotent with a given bound for the class.

This situation suggests an axiomatic approach. Let $G$ be any locally nilpotent group and let $\nu$ be a binomial ring such that $x^\lambda$ is a uniquely determined element of $G$ for any $x$ in $G$ and $\lambda$ in $\nu$.

\begin{defn*} $G$ is a $\nu$-\underline{powered} group if the following axioms hold:
\begin{itemize}
\item[(I)] $x^1 = x; x^{\lambda + \mu} = x^{\lambda} \cdot  x^{\mu}, \text{ and } x^{\lambda \mu} = (x^\lambda)^{\mu}.$
\item[(II)] $y^{-1} x^\lambda y = (y^{-1} x y)^\lambda$
\item[(III)] $$x_1^{\lambda} x_2^{\lambda} \ldots x_n^{\lambda} = t_1^{\lambda} t_2^{\lambda \choose 2} \ldots t_c^{\lambda \choose c},$$ where c is the class of the nilpotent group $\left<x_1, \ldots, x_n\right>$ and $t_k = t_k(x_1,\ldots, x_n)$.
\end{itemize}
\end{defn*}

In (I), $1$ is the unit of $\nu$. The $x's$ and $y$ are arbitrary elements of $G$; $\lambda$ and $\mu$ arbitrary elements of $\nu$. Finally, (III) is supposed to hold for all finite $n$. 

If $G$ happens to be Abelian, these axioms reduce to the usual axioms for $R$-modules. For by (\ref{6.3}), $t_1, t_2, \ldots$ all lie in $G^\prime = \left<1\right>$. And (II), which states that exponentiation by any element of $\mu$ commutes with every inner automorphism of $G$, becomes vacuous, since the only inner automorphism of an abelian group is the identity. 

It follows from (I), (II), (III) that
\begin{equation}\label{banana}
1^\lambda = 1; \quad x^0 = 1; \quad x^{\lambda} = (x^\lambda)^{-1}\tag{IV}
\end{equation}
where $1$ is the unit of $G$ and $0$ is the zero of $\nu$. It is also clear that, if $n$ is an integer considered as an element of $\nu$, then $x^n$ has the usual meaning. For a given $a \in G$, the mapping $\lambda \to a^\lambda$ maps the additive group of $\nu$ homomorphically onto the subgroup $a^\nu$ of $G$ and the kernel $\mathcal{P}$ of this map is an ideal of $R$, the \underline{order-ideal} of $a$. $G$ may be called $\nu$-\underline{torsion-free} if this order ideal is $0$ for all $a \neq 1$ in $G$.

Define an $\nu$-homomorphic mapping (of one $\nu$-powered group into another) and $\nu$-powered subgroup in the obvious wasy.
\begin{lemma}\label{6.6}
The kernel of an $\nu$-homomorphism is a normal $\nu$-powered subgroup. If $K$ is any normal $\nu$-powered subgroup $G$, then $$x \equiv y \text{ mod } K \quad \text{ implies that } \quad x^{\lambda} = y^{\lambda} \text{ mod } K$$ for all $\lambda \in \nu$. $G/K$ can be interpreted as an $\nu$-powered group. The natural homomorphism of $G$ onto $G/K$ is a $\nu$-homomorphism. The usual isomorphism theorems of group theory hold.
\end{lemma}
\begin{proof}
For suppose that $$x \equiv y \text{ mod } K.$$ Let $c$ be the class of the group $\left<x,y\right>$, and let $t_k = \tau_k(x,y^{-1})$. For integers $n > 0$, we have the identities
$$
x^n y^{-n} = t_1^n t_2^{n \choose 2} \ldots t_n.
$$
This identity is valid in any group, in particular in $G/K$. But $$x^n \equiv y^n \text{ mod } K$$ for all $n = 1, 2, \ldots \: .$ Hence $x^n y^{-n} \in K$, and so by induction, each $t_n \in K$. By axiom (III), we have also
$$
x^\lambda y^{-\lambda} = t_1^{-\lambda} t_2^{\lambda \choose 2} \ldots t_c^{\lambda \choose c}.
$$
Since $K$ admits exponentiation by the elements $\lambda \choose i$ of $\nu$, and each $t_i \in K$, we conclude that $x^\lambda y^{-\lambda} \in K$, $$x^\la \equiv y^\la \text{ mod } K.$$ The rest is clear.
\end{proof}

Examples of $\nu$-powered groups.
\begin{enumerate}
\item Let $G$ be a torsion free, finitely generated nilpotent group. Then $G^\nu$ is a $R$-powered group.
\item Let $R$ be any ring (associative, with unit element) which contains $\nu$ in its centre. Let $\Fr{k}$ be a nilpotent ideal of $\nu$, say $\Fr{k}^{n+1} = 0$. Then $G = 1 + \Fr{k}$ is nilpotent of class $\leq n$, by (\ref{theorem3.6}). For $\lambda \in \nu$ and $u \in \Fr{k}$, we define 
$$
(1 + u)^\lambda = 1 + \lambda u + {\lambda \choose 2}u^2 + \ldots + {\lambda \choose n}u^n \in G.
$$ This definition makes $G$ a $\nu$-powered group.
\end{enumerate}

Another cases is the group $T_n(\nu)$ of all upper triangular matrices with $1$'s on the diagonal and coefficients in $\nu$.

Another interesting case is obtained by considering the group ring $\nu G$ of an arbitrary group $G$ over the ring $\nu$. $\nu G$ consists of linear forms $$u = \sum_{x \in G} \lambda_x x,$$ with coefficients $\lambda_x$ in $\nu$ and all but a finite number of them zero, addition and multiplication being defined in the natural way. Those $u$ in $\nu G$ for which $$\sum_{x \in G} \lambda_x = 0$$ form an ideal $\Delta$, the \underline{difference ideal.} \footnote{or augmentation ideal} Take $R = \nu G / \Delta^{n+1}.$ Every $x$ in $G$ is $\equiv 1 \! \! \! \mod \Delta.$ If there should exist an integer $n$ such that $G \cap 1 + \Delta^{n+1} = 1,$ then $G$ is embedded in a natural way in $R$ and more precisely in the $\nu$-powered group $1 + \Delta/\Delta^{n+1}.$ It can be shown that if $G$ is any torsion-free nilpotent group of class $\leq n,$ then in fact $G \cap 1 + \Delta^{n+1} = 1.$ Thus we obtain an embedding of any such group in a nilpotent $\nu$-powered group of the same class by a new method.

\begin{thm}\label{6.7}
Let $G$ be a torsion free, finitely generated and nilpotent. Let $H$ be an arbitrary $\nu$-powered group. Then any homomorphic mapping of $G$ into $H$ can be extended to an $\nu$-homomorphic mapping of $G^\nu$ into $H$.
\end{thm}

\begin{cor*} Let $X$ be any subgroup of the $\nu$-powered group $H$. Then $\left<X^H\right>$ is the smallest $\nu$-powered subgroup of $H$ that contains $X$.\end{cor*}

This and many other properties of $\nu$-powered groups have been developed recently by A. Duguid. In particular, the theory of isolators of $\S4.$ may be largely extended from the case $\nu =$ ring of rational integers (as in $\S4.$) to the more general case.

\section{Dimension Subgroups}
\begin{defn*} Let $G$ be any group, $\mathbb{K}$ any field of characteristic zero and $\Delta = \Delta(\mathbb{K},G)$ the difference ideal of the group ring $\mathbb{K}G$ of $G$ over $\mathbb{K}$. Then  $$\delta_n(G) = G \cap 1 + \Delta^n$$ is called the $n^{th}$ \underline{dimension subgroup} of $G$ (with respect to $\mathbb{K}$). Clearly $\delta_1(G) = G.$ Also $\delta_n(G) \nsub G$ since it is the kernel of the homomorphism induced on $G$ by the natural homomorphism of $\mathbb{K}G$ onto $\mathbb{K}G / \Delta^n.$
\end{defn*}

In fact, $\delta_n(G)$ is independent of the choice of $\mathbb{K}$, because of
\begin{thm}[Jennings]\label{7.1}
An element $x$ of $G$ belongs to $\delta_n(G)$ if and only if $x^m \in \gamma_n(G)$ for some $m>0$
\end{thm}
Since $G / \gamma_n(G)$ is nilpotent, the periodic elements of this factor group form a subgroup $\bar{\gamma}_n(G) / \gamma_n(G)$, by (\ref{theorem4.5}(d)). The theorem of Jennings is that $\delta_n(G) = \bar{\gamma}_n(G)$ for all groups $G$ and all integers $n > 0.$\\

\noindent \emph{Proof.}
(1) $\bar{\gamma}_n(G) \leq \delta_n(G)$. Let $x$ be any element of $G$ not in $\delta_n(G)$. Then $u = 1-x$ belongs to $\Delta$ but not to $\Delta^n$. Suppose $u$ belongs to $\Delta^k$ but not to $\Delta^{k+1}$. Here $1 \leq k < n.$ For $m>0,$ $$x^m = 1- mu + {m \choose 2}u^2 \ldots$$ and $u^2, u^3, \ldots$ all lie in $\Delta^{k+1}.$ Hence 
$$
1- x^m \equiv mu \text{ mod } \Delta^{k+1}
$$
and, since $\mathbb{K}$ has characteristic zero, it follows that $x^m \notin \delta_n(G)$ for any positive $m$. But by (\ref{theorem3.6}), $G / \delta_n(G)$ is nilpotent of class $< n$, and so $\gamma_n(G) \leq \delta_n(G).$ Hence $\bar{\gamma}_n(G) \leq \delta_n(G).$

(2) We may assume that $G$ is finitely generated. For $\bar{\gamma}_n(G)$ is the union of all the $\bar{\gamma}_n(H)$ where $H$ runs over the finitely generated subgroups of $G$; and the same remark applies to $\delta_n(G)$. For $\Delta$ is a vector space over $\mathbb{K}$ with a basis consisting of all $1-x$ with $x \neq 1$ in $G$. Hence $\Delta^n$ is the $\mathbb{K}$-space spanned by all products of the form $(1-y_1)(1-y_2)\ldots (1-y_n)$ with $y$'s in $G$. Thus $x \in \delta_n(G)$ if and only if $1-x$ is expressible as a (finite) linear combination of such products, and such an expression can only involve a finite number of elements $y$.

\begin{defn*} 
A $\mathcal{J}$-\underline{group} is a torsion-free and finitely generated nilpotent group.
\end{defn*}

(3) We may assume that $G$ is a $\mathcal{J}$-group. For let $K = \bar{\gamma}_n(G)$. Then $G/K$ is torsion free and by (\ref{lemma1.2}), $\gamma_n(G/K) = 1$. Hence $\bar{\gamma}_n(G/K) = 1$. But the natural homomorphism of $G$ onto $G/K$ induces a homomorphism of $\mathbb{K}G$ onto $\mathbb{K} (G/K)$ in which $\Delta^n(\mathbb{K},G)$ maps onto $\Delta^n(\mathbb{K}, G/K).$ Hence $\delta_n(G/K) = K \delta_n(G) / K$. But $K \leq \delta_n(G)$ by (1). Thus we may assume $K =1$ and have to prove that $\delta_n(G) = 1.$ In conjunction with (2), this means that $G$ is a $\mathcal{J}$-group. 

\begin{defn*} A basis $(u_i)$ of the vector space $\mathbb{K}G$ is called an \underline{integral basis} if
\begin{itemize}
\item[(i)] each $u_i$ has the form $\sum_{x \in G} n_x x$ with rational integers $n_x$ and conversely
\item[(ii)] every $x \in G$ has the form $\sum_i m_i u_i$ with integers $m_i$.
\end{itemize}
\end{defn*}

Let $u_i u_j = \sum_k c_{ij}^k u_k$. If $(u_i)$ is an integral basis of $\mathbb{K}G$, then the multiplication constants $c_{ij}^k$ are all integers.

\begin{lemma}\label{7.2}
Let $G = \left<x\right>$ be an infinite cyclic group, $n$ any positive integer and $u = 1-x$. Then $1, u, u^2, \ldots$ together with $u^n x^{-1}, u^n x^{-2}, \ldots$ form an integral basis of $\mathbb{K}G.$
\end{lemma}
\begin{proof}
For 
$$
u^r = \sum_{s=0}^r(-1)^s {r \choose s}x^s \quad (r=1,2,3, \ldots),
$$
which shows that $1, u, u^2, \ldots$ together with $x^{-1}, x^{-2}, x^{-3}, \ldots$ for an integral basis. This is the case $n=1$. But $$u^{n+1}x^{-1} = u^n x^{-r} - u^nx^{-r+1}$$ which gives the induction from $n$ to $n+1$.
\end{proof}

Now let $G$ be a $\mathcal{J}$-group and write $\bar{\Gamma}_i = \bar{\gamma}_i(G)$. We have to show that $\bar{\Gamma}_n = 1$ implies that $\delta_n(G) = 1$ and so we may assume without loss of generality that $n = c +1$ where $c$ is the class of $G$. Thus $$G = \bar{\Gamma}_1 > \bar{\Gamma}_2 > \ldots > \bar{\Gamma}_{c+1} = 1$$ and each $\bar{\Gamma}_{i-1} / \bar{\Gamma}_i$ is a free Abelian group of finite rank $m_i.$ Also, by (\ref{4.6}) corollary, the $\bar{\Gamma}$-series is a central series of $G$. We may therefore refine the $\bar{\Gamma}$ to a finite series $$G = G_0 > G_1 > \ldots > G_m = 1$$ of $G$ with each $G_{i-1} / G_i$ infinite cyclic; $M = m_1 + m_2 + \ldots + m_c$. Choose $x_i$ so that $G_{i-1} = \left<G_i, x_i\right>$. Then $x_1, x_2, \ldots, x_M$ will be a canonical basis of $G$ and every element of $G$ is uniquely expressible in the form $x_1^{r_1} x_2^{r_2} \ldots x_M^{r_M}$. Let $u_i = 1-x_i.$ Then (\ref{7.2}) gives

\begin{lemma}\label{7.3}
The set of all products $v = v_1 v_2 \ldots v_M,$ in which each $v_i$ has one of the forms $u_i^{r_i}$ $(r_i \geq 0)$ or else $(u_i^{n}x_i^{-s_i})$ $(s_i > 0),$ is an integral basis of $\mathbb{K}G.$
\end{lemma}

\begin{defn*}
$\mu(u_i) = k \Leftrightarrow x_i$ is in $\bar{\Gamma}_k$ but not in $\bar{\Gamma}_{k+1}$. For any of the products $v$, define $\mu(v) \geq n$ whenever at least one of the factors $v_i$ has the form $u_i^n x_i^{-s_i}$; but otherwise let 
$$
\mu(v) = \mu(u_1^{r_1} u_2^{r_2} \ldots u_M^{r_M}) = \sum_{i=1}^M r_i \mu_i
$$ where $\mu_i = \mu(u_i).$
\end{defn*}

Denote by $E_s,$ where $s \leq n,$ the $\mathbb{K}$-space spanned by all those products $v$ which have $\mu(v) \geq s$. For $s>n$, define $E_s = E_n$.

(4) It is sufficient to prove that $\Delta^s \leq E_s$ if $s \leq n.$ For $\bar{\Gamma}_k \leq \delta_k(G)$ by (1). Hence those $u_i$ for which $\mu_i = \mu(u_i) = k$ all belong to $\Delta^k.$ If $s \leq n$ and $\mu(v) \geq s$, it follows that $v \in \Delta^s$, and so $E_s \leq \Delta^s$. Combined with $\Delta^s \leq E_s$; this gives $\Delta^s = E_s$ for $s \leq n$ and therefore those $u_i$ for which $\mu_i = k$ are linearly independent $\text{mod } \Delta^{k+1}.$ Suppose these $u_i$ are precisely $u_j, u_{j+1}, \ldots, u_{j+ \ell};$ and let $y \in \bar{\Gamma}_k - \bar{\Gamma}_{k+1},$ so that $k < n.$ Then $y = x_j^{r_0}x_{j+1}^{r_1} \ldots x_{j+\ell}^{r_\ell z}z$ with $z \in \bar{\Gamma}_{k+1}$ and $r_0, r_1, \ldots, r_\ell$ not all zero. For $0 \leq \alpha \leq \ell,$ we have 
$$
1- x_{j+\alpha}^{r_\alpha} \equiv r_\alpha u_{j + \alpha} \text{ mod } \Delta^{k+1}
$$
and $1-z \in \Delta^{k+1}.$ Hence 
$$
1 - y = \sum_{\alpha} r_{\alpha} u_{j+\alpha} \neq 0 \text{ mod } \Delta^{k+1},
$$ and so $y \notin \delta_{k+1}(G).$ In particular, no element $y \neq 1$ of $G$ lies in $\delta_n(G)$.

(5) It is enough to prove $E_r E_s \leq E_{r+s}$. For $E_1 = \Delta$ and (5) then gives $\Delta^s \leq E_s$, for all $s$, by induction.

For $1 \leq i \leq M,$ let $E_{s,i}$ be the subspace of $E_s$ spanned by all those products $v = v_1 v_2 \ldots v_M$ which lie in $E_s$ and have $v_1 = v_2 = \ldots = v_{i-1} = 1$. For $s>n$, we have $E_{s,i} = E_{n,i}$. In any case $E_{s,i}$ is contained in $\mathbb{K}G_{i-1}.$ Also $E_{s,1} = E_s.$ Hence

(6) It is enough to prove $E_{r,i} E_{s,i} \leq E_{r+s,i}.$ This is immediate when $i=M.$ Let $i<M$ and use induction on $M-i$. Thus we may assume that $E_{r,j} E_{s,j} \leq E_{r+s,j}$ for $j > i.$ From this we first deduce

\begin{lemma}\label{7.4}
Let $v_1 v_2 \ldots v_m \in E_p$. Then $$vx_i \equiv v \text{ mod } E_{p + \mu_i,i}.$$
\end{lemma}
\begin{proof}
For 
$$
x_i^{-1} u_j x_i = 1- x_j[x_j,x_i] = u_j +z - u_j,
$$
where $z = 1 - [x_j, x_i].$ But 
$$
[x_j, x_i] \in [\bar{\Gamma}_{\mu_j}, \bar{\Gamma}_{\mu_i}] \leq \bar{\Gamma}_{\mu_i + \mu_j}
$$ 
by (\ref{4.6}) corollary. Hence $z$ belongs to the difference ideal of $\mathbb{K} \bar{\Gamma}_{\mu_i + \mu_j}$ and so $z \in E_{\mu_i + \mu_j, i+1}.$ By our assumption, if $j>i$, we also have $u_j z \in E_{\mu_i + \mu_{j,i+1}}.$ Hence 
$$x_i^{-1} \mu_j x_i \equiv u_j \text{ mod } E_{\mu_i + \mu_{j, i+1}}.
$$ It follows that 
$$
x_i^{-1} \mu_j^r x_i \equiv u_j \text{ mod } E_{\mu_i + \mu_{j,i+1}};
$$
in particular, $x_i^{-1}u_j^{n}x_i \in E_{n, i+1};$ also $x_i^{-1} x_j^{-s} x_i \in E_{0,i+1}$ and so 
$$
x_i^{-1} u_j^{n}x_j^{-s} x_i \equiv u_j x_j^{-s} \text{ mod } E_{n,i+1},
$$  since both terms are $\equiv 0.$ Thus 
$$
x_i^{-1}v_j x_i \equiv v_j \text{ mod } E_{\mu_i + \mu(v_j),i+1}
$$
for each $j>i.$ Put 
$$
vx_i = v_1 v_2 \ldots v_i(1-u_i) \prod_{j=i+1}^M (x_i^{-1}v_j x_i),
$$ and (\ref{7.4}) now follows, since $\mu(u_i) = \mu_i.$

By (\ref{7.4}), $vu_i = v(1-x_i) \in E_{p + u_i,i}$ if $v \in E_{p,i}.$ Hence $vu_i^r \in E_{p + r\mu_i,i}$ for any $r>0$. In particular, $vu_i^n \in E_{n,i}$. And so, by an argument similar to the proof of (\ref{7.4}) but with $x_i^{-1}$ for $x_i$, $vu_i^nx_i^{-s} \in E_{n,i}.$ Hence if $v' =  v_i' v_{i+1}' \ldots v_M' \in E_{q,i},$ we have $v v_i' \in E_{p + \mu(v_i'),i}$ and therefore by the induction hypothesis $vv' \in E_{p+q,i}$ as required. Thus (6) holds for all $i$ and the proof of (\ref{7.1}) is complete.
\end{proof}

Another result now appears incidentally. Let $G$ be any $\mathcal{J}$-group of class $c$, and let $m_i$ be the rank of $\bar{\gamma}_i(G) / \bar{\gamma}_{i+1}(G),$ $i=1,2, \ldots, c$. Then the linear space $\Delta^j / \Delta^{j+1},$ where $\Delta = \Delta(\mathbb{K},G)$ and $\mathbb{K}$ is a field of characteristic zero, has dimension $d_j$ equal to the number of products $v = u_1^{r_1}u_2^{r_2} \ldots u_M^{r_M}$ of weight $\mu(v) = j.$  Hence the identity
$$
\sum_{j=0}^{\infty} d_jt^j = \frac{1}{(1-t)^{m_1}(1-t^2)^{m_2} \ldots (1-t^c)^{m_c}}.
$$
Here $d_0 = 1 = \dim \mathbb{K}G / \Delta$ if we define $R = \mathbb{K}G / \Delta^{c+1}$, then $G$ is mapped isomorphically into $R$ in the obvious way. $\text{dim} \: R = d_0 + d_1 + \ldots d_c = d$ say. Let $b_1, b_2, \ldots, b_d$ be the $d$ products $v$ with weight $\leq c$ arranged in order of increasing weight, but otherwise arbitrarily. By (\ref{7.4}), $$b_ix_j \equiv b_i + \text{ terms in } b_{i+1}, \ldots, b_d \text{ mod } \Delta^{c+1}.$$ But $x_j = 1-u_j$ and the products $v$ form an integral basis of $\mathbb{K}G$. Hence the coefficients which occur in this expression for $b_i x_j$ are all integers. Thus we obtain a faithful representation of $G$ by unitriangular $d \times d$ matrices with rational integer coefficients giving
\begin{thm}\label{7.5}
Every $\mathcal{J}$-group $G$ is isomorphic with a subgroup of $T_d(\nu_0)$, where $\nu_0$ is the ring of rational integers $d = d(G)$ is chosen suitably.
\end{thm}

\begin{note*} A finite $p$-group $G$ can be embedded in $T_g(v_p)$ where $g = |G|$ and $v_p$ is the field of $p$ elements. For if $\underline{M}$ is the Magnus ideal of the group algebra $\nu_p G,$ we have $\underline{M}^{c+1} = 0$ and so the regular representation of $G$ over $\nu_p$ is unitriangular.
\end{note*}

Let $U_d$ be the unimodular group of degree $d$, that is the group of all $d \times d$ matrices with integer coefficients and determinant $\pm 1.$ It seems probable that every polycylic group is isomorphic with a subgroup of some $U_d$. \footnote{That this is so was proved by L. Auslander: On a problem of Philip Hall, Annals of Math. 86 (1967), 11-116. A simpler proof was then given by R. G. Swan: Representations of polycyclic groups, Proc. American Math. Soc. 18 )1967), 513.}

\begin{lemma}\label{7.6}
Let $H$ be a subgroup of $G$ such that $|G:H| = m$ is finite. Then, if $H$ can be embedded in $U_d$,  $G$ can be embedded in $U_{md}$.
\end{lemma}
\begin{lemma}\label{7.7}
A finitely generated and periodic soluble group is finite.
\end{lemma}
\begin{proof}
Proof by induction on the length of the derived series, using the fact that a subgroup of finite index in a finitely generated group is itself finitely generated.
\end{proof}

\begin{thm}[Hirsch]\label{7.8}
Let $G$ be polycyclic and let $x$ be any element $\neq 1$ in $G$. Then, for some $m > 0$, $G^m$ is torsion-free and does not contain $x.$ 
\end{thm}
\begin{proof}
Note that $G^m$ is a characteristic subgroup of $G$. Also, since $G$ is soluble and finitely generated, $[G:G^m]$ is finite by (\ref{7.7}). Suppose $$G = G_0 \trianglerighteq G_1 \trianglerighteq \ldots \trianglerighteq G_r = 1,$$ with cyclic factors $G_{i-1} /G_i$. When $r=1,$ $G$ is cyclic and the result is immediate.

Let $r>1$. Then we may assume that $G_1^n$ is torsion-free for some $n>0$ and does not contain $x$. If $[G:G_1]$ is finite, then $|G: G_1^n| = m$ is also finite. Also $G_1^n \trianglelefteq G$ and so $G^M \leq G_1^n$. Then $G^m$ is torsion-free and doesn't contain $x$. If $[G:G_1]$ is finite, let $G = \left<G_1, a\right>$. Then $a$ induces in the finite group $G_1 / G_1^n$ an automorphism of some finite order $t$.

Also $x=a^\lambda y$ with $y$ in $G_1$. In case $\lambda \neq 0,$ choose a multiple $u$ of $t$ such that $\lambda$ is not a multiple of $u$. Then, since $[G_1, a^u] \leq G_1^n$ by our choice of $u$, the group $H= \left<a^u, G_1^n\right>$ is normal in $G$ and of index $u|G_1 : G_1^n| = m$, say. Also $x \notin H;$ and $H$ is torsion-free since $G_1^n$ is, and $H/G_1^n$ is infinite cyclic. Then $G^m$ is contained in $H$ and therefore satisfies the requirements of Hirsch's theorem.
\end{proof}

\begin{cor*} Every polycyclic group $G$ which contains a nilpotent subgroup $H$ of finite index can be embedded in some $U_d$. In particular, every supersoluble group can be embedded.\end{cor*}
\begin{proof}
For if we choose $m$ so that $G^m$ is torsion-free, then $K = H \cap G^m$ is a $\mathcal{J}$-group and $[G:K]$ is finite. The corollary now follows from (\ref{7.5}) and (\ref{7.6}).
\end{proof}

\begin{thm}\label{7.9}
The holomorph of any $\mathcal{J}$ can be embedded in some $U_d.$
\end{thm}

\section{Theorems of Schur and Baer}
\begin{thm}[Schur]\label{8.1}
Let $|G:Z_1| = m$ be finite. Then, for $x$ in $G$, the map $x \to x^m$ is homomorphic and $|G'|$ is an $m$-number
\end{thm}
\begin{proof}
Proof by transfer. Let $\tau$ be the transfer of $G$ into its centre $Z_1$. For any $x$ in $G$, $\tau(x)$ has the form 
$$
\prod_i s_i x^{\rho_i} s_i^{-1} \quad \text{ with } \quad \sum_i \rho_i = m
$$
and each $s_i x^{\rho_i}s_i^{-1} \in Z_1.$ Hence $s_i x^{\rho_i} s_i^{-1} = x^{\rho_i}$ and $\tau(x) = x^m$. But $\tau$ is homomorphic.

Since $\tau(G) \leq Z_1$, it is Abelian and the kernel $K$ of $\tau$ contains $G'$. Hence $y^m = 1$ for all $y \in G'.$ Let $L = G' \cap Z_1.$ Since $[x_1, x_2] = [x_1', x_2']$ whenever 
$$
x_1 \equiv x_1' \quad \text{ and } \quad x_2 \equiv x_2' \mod Z_1,
$$ 
$G'$ is finitely generated. But $|G'/L| = |Z_1 G' / Z_1|$ which is finite and an $m$-number. Hence $L$ is a finitely generated Abelian group whose elements $z$ all satisfy $z^m = 1.$ Therefore $|L|$ is an $m$-number, and so $|G'|$ is an $m$-number.
\end{proof}

\begin{lemma}[Baer]\label{8.2}
Let $H \nsub G$, let $K$ be a subgroup of $H$ such that $[K,G]=1$ and let $|H:K| = m$ be finite. Then $[H,G]^m =1.$
\end{lemma}
\begin{proof}
Let $\tau$ be the transfer of $H$ into $K$. Since $K$ is contained in the centre of $H$, $\tau(x) = x^m$ for $x \in H$, as in the proof of (\ref{8.1}). Let $y \in G$. Then $$\tau([x,y]) = \tau(x^{-1} \cdot x^y) = \tau(x^{-1}) \tau(x^y) = x^{-m}(x^y)^m = [x^m, y] = 1$$ since $x^m \in K$ and $[K,G]=1.$ But $\tau([x,y]) = [x,y]^m.$ Hence $[x,y]^m = 1.$ All generators $[x,y]$ of $[H,G]$ belong to the kernel of $\tau$, so $[H,G]^m = 1.$
\end{proof}

\begin{thm}[Baer]\label{8.3}
Let $H,$ $H_1,$ $K,$ $K_1$ be normal subgroups of $G$ such that $H \leq H_1,$ $K \leq K_1$ and $|H_1:H| = m$ and $|K_1:K| = n$ are finite. Then $|[H_1, K_1]:[H_1, K][H, K_1]|$ is a $m$-number.
\end{thm}
\begin{proof}
We may suppose $[H_1,K] = [H, K_1] = 1$ since these are normal subgroups of $G$. Choose representatives $u_\rho$, $v_\alpha$ for the cosets $\rho$ of $H$ in $H_1$ and the cosets $\alpha$ of $K$ in $K_1$. Every $y_1 \in K_1$ has the form $yv_\alpha$ for some $y \in K$ and some $\alpha.$ If $x_1 \in H_1$, then $[x_1,y] =1$ and so $[x_1, y_1] = [x_1, v_\alpha]$. But $x_1 = xu_\rho$ for some $x \in H$ and some $\rho$. Since $[x,v_\alpha] = 1$ we have $[x_1, v_\alpha] = [u_\rho, v_\alpha].$ Hence $D = [H_1, K_1]$ is finitely generated by the elements $[u_\rho, v_\alpha].$

Let $E = D \cap H.$ Since $D \leq H_1,$ $|D:E| = m'$ divides $|H_1: H| =m$. Hence $E$ is finitely generated. Since $E \leq H,$ $[E,K_1] = 1.$ Also $D \leq K_1$ (so $E$ is Abelian). By (\ref{8.2}) with $K_1, D, E$ for $G,H,K$, we have $[D,K_1]^{m'} = 1.$ But $[D,K_1]$ contains $E \cap [D,K_1]$ as a finitely generated Abelian subgroup of finite index. Therefore $|[D,K_1]|$ is an $m$-number. Since $[D, K_1] \nsub G,$ we may suppose $[D,K_1] = 1.$ Hence $D$ being $\leq K_1$ is Abelian.

Then $v_\beta$ commutes with $[u_\rho, v_\alpha]$ and hence $[u_\rho, v_\alpha v_\beta] = [u_\rho, v_\alpha][u_\rho, v_\beta]$. Therefore $\alpha \to [u_\rho, v_\alpha]$ is homomorphic and, for fixed $\rho \in H_1/H,$ the elements $[u_\rho, v_\alpha]$ form a group $D_\rho$ whose order divides $n$. But $D \leq K_1$ and $[D,K_1] =1$ so $D$ is Abelian. Since $D = \prod_\rho D_\rho$, $|D|$ is an $n$-number.
\end{proof}

\begin{defn*} Let $$\phi = \phi(x_1, \ldots, x_n)$$ be any word in the variables $x_i$, and $G$ any group. Then $\phi^*(G)$ is to consist of all elements $a$ such that 
$$
\phi(x_1, \ldots, x_{i-1}, x_ia, x_{i+1}, \ldots, x_n) = \phi(x_1, \ldots, x_n)
$$
for each $i = 1,2, \ldots, n$ and all choices $x_1, \ldots, x_n$ in $G$. We then call the subgroup $\phi^*(G)$ the \underline{marginal subgroup} of $G$ corresponding to the verbal subgroup $\phi(G)$ generated by all the values of $\phi$ in $G$.
\end{defn*}

\begin{lemma}\label{8.4}\:
\begin{itemize}
    \item[(i)] $\phi^*(G)$ is a characteristic subgroup of $G$. If $$x_i \equiv y_i \text{ mod } \phi^*(G)$$ for each $i=1,2, \ldots,n$, then $\phi(x) = \phi(y)$ and $\phi^*(G)$ is the largest normal subgroup of $G$ with this property.
    \item[(ii)] $\phi(\phi^*(G)) = 1$. In particular, $\phi(G) = 1 \Leftrightarrow \phi^*(G) = G.$
    \item[(iii)] If $H \nsub G$ and $H \cap \phi(G) = 1,$ then $H \leq \phi^*(G).$
    \item[(iv)] If $K / \phi^*(G)$ is the centre of $G / \phi^*(G)$, then $[K, \phi(G)]=1.$ In particular, $[\phi^*(G), \phi(G)] = 1.$
    \item[(v)] $\phi(G \times H) = \phi(G) \times \phi(H)$, $\phi^*(G \times H) = \phi^*(G) \times \phi^*(H)$.
    \item[(vi)] $H\phi^*(G) = G$ implies $\phi(H) = \phi(G).$
    \item[(vii)] If $K_\lambda \nsub G,$ $$\bigcap_{\lambda \in \Lambda} K_\lambda = 1,$$ then $\phi^*(G) = \bigcap_{\lambda \in \Lambda} \phi^*(G \text{ mod } K_\lambda)$.
\end{itemize}
\end{lemma}
\begin{proof}(i) is clear. (ii) If $a_1, \ldots, a_n$ belong to $\phi^*(G)$, we have $$\phi(a_1, \ldots, a_n) = \phi(1,1, \ldots, 1) = 1,$$ by (i). 
    
    (iii) Let $a \in H.$ Since $H \nsub G,$ the element $$u_i = (\phi(x))^{-1}\phi(x_1, \ldots, x_ia, \ldots, x_n)$$ belongs to $H$. But $u_i \in \phi(G).$ Hence $u_i = 1$ for all $i$ and all choices of $x_1, \ldots, x_n$ in $G$. Thus $H \leq \phi^*(G)$. 
    
    (iv) Let $y \in K$. Then 
    $$
    y^{-1} \phi(x) y = \phi(x_1[x_1,y], \ldots, x_n[x_n,y]) = \phi(x)$$ since each $[x_i, y] \in \phi^*(G).$ Hence $[K, \phi(G)]=1.$ \end{proof}

\begin{lemma}\label{8.5}
Let $\theta(x_1, \ldots, x_m)$ and $\phi(x_1, \ldots, x_n)$ be any two words and let $$\psi(x_1, \ldots, x_{m+n}) = [\theta(x_1, \ldots, x_m), \phi(x_{m+1}, \ldots, x_{m+n})].$$ Let $G$ be any group and let $U, V$ be the centralisers in $G$ of $\theta(G)$ and $\phi(G).$ Then $\psi^*(G) = H \cap K$ where $H/V = \theta^*(G/V)$ and $K/U = \phi^*(G/U).$
\end{lemma}
\begin{proof}
Write $y_j$ for $x_{m +j}.$ Let $a \in \psi^*(G).$ Then, for all choices of $x_i's$ and $y_j's$ in $G$, we must have 
$$
[\theta(x), \phi(y)] = [t_i \theta(x), \phi(y)]
$$
where 
$$
t_i = \theta(x_1, \ldots, x_i a, \ldots, x_m) (\theta(x))^{-1}.
$$
Hence $[t_i, \phi(y)] = 1$ for all $y$; so $t_i \in V$ for all $i = 1,2, \ldots,m.$ Therefore $a \in H.$ Similarly $a \in K.$ Conversely if $a \in H \cap K,$ we have $t_i \in V$ for all $i$ and all choices of the $x_j$ in $G$. Hence $[\theta(x), \phi(y)]$ is unaffected when $x_i$ is replaced by $x_ia.$ Similarly, it is unaltered when $y_j$ is replaced by $y_ja.$ Thus $a \in \psi^*(G).$
\end{proof}
\begin{lemma}\label{8.6}
$\gamma_{k+1}^*(G) = \zeta_k(G)$ for all groups $G$.
\end{lemma}
\begin{proof}
This clear when $k=0.$ Take $\theta = \gamma_k$ and $\phi = \gamma_1$ in (\ref{8.5}). Then $\psi = \gamma_{k+1}.$ Suppose $\gamma_k^*(G) = \zeta_{k-1}(G).$ Then by (\ref{8.4}(iv)), $U \geq \zeta_k(G)$ and so $K \geq \zeta_k(G)$. But $V= \zeta_1(G)$ and hence $H = \zeta_k(G).$ Thus $\gamma_{k+1}^*(G) = H \cap K =\zeta_k(G).$ The result follow from induction on $k$.
\end{proof}

\begin{defn*} The word $\phi$ has the \underline{Schur-Baer property} if, in any group $G$ for which $|G:\phi^*(G)| = m$ is finite, $|\phi(G)|$ is always an $m$-number. 
\end{defn*}
Schur's theorem (\ref{8.1}) states that $\gamma_2$ has the Schur-Baer property. The conjecture is that all words have this property.

\begin{thm}[Baer]\label{8.7}
If $\theta$ and $\phi$ have the Schur-Baer property, so does the word $\psi$ of (\ref{8.5}).
\end{thm}
\begin{proof}
Let $G$ be any group, let $H_1 =  \theta(G), K_1 = \phi(G),$ let $U$ and $V$ be the centralisers of $H_1$ and $K_1$ in $G$, and let $H = H_1 \cap V$ and $K = K_1 \cap U.$ By (\ref{8.5}) $\psi^*(G) = L \cap M$ where $L/M = \theta^*(G/V)$ and $M/U = \phi^*(G/U).$ Suppose $|G:\psi^*(G)| = m$ is finite. Then $|G:L| = |G/V : \theta^*(G/V)|$ divides $m$. Since $\theta$ has the Schur-Baer property, it follows that $|\theta(G/V)|$ is an $m$-number. But $\theta(G/V) = V \theta(G)/V$ and so $|\theta(G/V)| = |H_1:H|$. Similarly, $|K_1:K|$ is an $m$-number. Therefore, by (\ref{8.3}), $|[H_1,K_1]: [H_1, K][H, K_1]]$ is an $m$-number. But $[H_1,K] = [H,K_1] = 1,$ and $[H_1,K_1] = \psi(G)$ by (\ref{lemma1.5}). Thus $|\psi(G)|$ is an $m$-number, as required.
\end{proof}

A more trivial result is
\begin{lemma}\label{8.8}
If $\theta(x_1, \ldots, x_m)$ and $\psi(x_1, \ldots, x_n)$ are any two words with the Schur-Baer property, then 
$$
\chi(x_1, \ldots, x_{m+n}) = \theta(x_1, \ldots, x_m) \phi(x_{m+1}, \ldots, x_{m+n})
$$ 
also has that property. Also, $\chi(G) = \theta(G) \phi(G)$ and $\chi^*(G) = \theta^*(G) \cap \phi^*(G)$.
\end{lemma}

A more interesting result is
\begin{thm}\label{8.9}
Let the group $G$ have a polycyclic subgroup of finite index and let $\theta$ be any word such that $|G: \theta^*(G)| = m$ is finite. Then $|\theta(G)|$ is an $m$-number.
\end{thm}
\begin{cor*} Let $\chi$ be any word such that $G / \chi(G)$ is nilpotent for all groups $G$. Then $\chi$ has the Schur-Baer property.\end{cor*}

\begin{note*} Can now be extended with "Abelian-by-nilpotent" in place of "nilpotent" by our theorems about residual finiteness. \footnote{The whole question of words $\chi$ with the Schur-Baer property was discussed by P. W. Stroud in "On a property of verbal and marginal subgroups" (Proc. Cambridge Phil. Soc. (1965), 61, 41-48). He proves as his Theorem 2 that if $\chi$ defines an $LR\mathfrak{F}$-variety then $\chi$ has the Schur-Baer property. Since $L\frak{U}\frak{N}$ are LR$\mathfrak{F}$ by Hall's theorem on finitely generated Abelian by Nilpotent groups, it follows that $G / \chi(G)$ is Abelian by nilpotent for all groups $G$ then $\chi$ defines an $\frak{U}\frak{N}$-variety and so $\chi$ has the Schur-Baer property.}\end{note*}

For example, if in (\ref{8.8}) we take $\phi = \gamma_k$ for some $k$, then $\chi$ has the Schur-Baer property irrespective of the choice of $\theta.$

The proof of (\ref{8.9}) may be made to depend on Hirsch's Theorem (\ref{7.8}). We may use this to reduce the proof of (\ref{8.9}) to the case when $G$ is finite. In this case, one shows with the help of Sylow's theorem and a classical theorem of Schur that $G$ has a subgroup $H$ such that $G = H\theta^*(G)$ and $|H|$ is an $m$-number. But the values $\theta(a_1, \ldots, a_n)$ with $a_1, \ldots, a_n$ in $G$ depend only on the cosets of $\theta^*(G)$ from which the $a$'s are selected. Thus $\theta(G) \leq H$ and $|\theta(G)|$ is a $m$-number.

\section{Theorems of Malcev}
Besides his results on the embedding of $\mathcal{J}$-groups in Lie groups and the like, Malcev has made another important contribution by showing that polycyclic groups and also soluble linear groups, unlike soluble groups in general, are rather closely related to nilpotent groups. By a \underline{linear group}, we mean a group of non-singluar linear transformations of a finite-dimensional vector space $V$ over some commutative field $\mathbb{K}$. Equivalently, such a group may be regarded as a group of $n \times n$ matrices with coefficients in $\mathbb{K}$ and determinant $\neq 0.$ The number $n$ is called the \underline{degree} of the group.

\begin{lemma}\label{9.1}
Let $A$ be a maximal normal abelian subgroup of the nilpotent group $G$. Then $A$ coincides with its centraliser $C$ in $G$.
\end{lemma}
\begin{proof}
For suppose if possible that $C = C_0 > A$ and define $C_{i+1} = [C_i, G]$ for $i \geq 0$. Since $G$ is nilpotent and $C \leq G,$ there will be a smallest value of $i$ for which $C_{i+1} \leq A$. Choose any element $x$ of $C_i$ which is not in $A$. Since $x \in C$, $\left<x, A\right> = \bar{A}$ will be Abelian. But $[\bar{A}, G] \leq C_{i+1} A \leq A < \bar{A}$ and so $\bar{A}$ is normal in $G$. This is contrary to the hypothesis that $A$ is a maximal normal Abelian subgroup of $G$.
\end{proof}

\begin{lemma}\label{9.2}
Let $G$ be any group and $Z_k = \zeta_k(G).$ If $Z_1^m = 1,$ then $Z_{k+1}^m \leq Z_k$ for all $k$.
\end{lemma}
\begin{proof}
When $k=1$, this is a corollary of (\ref{lemma4.1}) applied with $X = Z_2$, $Y=G.$ For $k>1$, it follows by induction.
\end{proof}

\begin{lemma}\label{9.3}
Let $G$ be a nilpotent group of $n \times n$ matrices of determinant $1$ with coefficients in a field $\mathbb{K}$. If $Z_1 \leq \mathbb{K} \cdot \text{Id}_{n \times n},$ then $G$ is finite.
\end{lemma}
\begin{proof}
We may suppose that $\mathbb{K}$ is algebraically closed. By hypothesis, every element of $Z_1$ has the form $\lambda \cdot \text{Id}_{n \times n}$ where $\text{Id}_{n \times}$ is the unit matrix and $\lambda \in \mathbb{K}.$ Also $\lambda^n = 1$, so that $|Z_1| = m$ divides $n$ and, if the characteristic of $\mathbb{K}$ is a prime $p$, then $m$ is prime to $p$. If $c$ is the class of $G$, then $G = Z_c$ and so $G^{m^c} = 1$ by (\ref{9.2}). Let $H$ be any finitely generated subgroup of $G$. By (\ref{theorem1.10}) $H$ is finite. Let $A$ be any maximal normal Abelian subgroup of $H$. Since $\mathbb{K}$ is algebraically closed, and $|A|$ being an $m$-number is prime to the characteristic of $k$ if this is finite, it follows that $A$ may be transformed into a group of diagonal matrices. All the diagonal coefficients which occur are $m^c$-th roots of unity. Hence $|A|$ divides $m^{cn}$. By (\ref{9.1}), $A$ is its own centraliser in $H$. Therefore $H/A$ is isomorphic with a group of automorphisms an Abelian group of bounded order. Thus both $|A|$ and $|H:A|$ are bounded. Hence $G$ is finite.
\end{proof}

\begin{thm}[Malcev]\label{9.4}
Let $G$ be a soluble linear group of degree $n$ over the algebraically closed field $\mathbb{K}$. Then there exists a subgroup $H$ of $G$ and a non-singular $n \times n$ matrix $t$ over $\mathbb{K}$ such that $|G:H|$ is finite and that every matrix $x \in t^{-1}Ht$ is triangular (that is, $x_{ij} = 0$ for $i>j$).
\end{thm}
\begin{proof}
For $n=1$ this is clear. Let $n>1$ and suppose the theorem is true for groups of degree $<n$.

Consider first the case when $G$ is \underline{reducible}, that is, when there is a matrix $t_0$ such that, for all $y \in G,$ 
$$
t_0^{-1}yt_0 = \begin{bmatrix} y_1 & *\\ 0 & y_2 \end{bmatrix}
$$ where each $y_1$ is an $r \times r$ matrix and therefore each $y_2$ is an $s \times s$ matrix, with $s = n-r$ and $0 < r < n$. The mappings $y \to y_1$ and $y \to y_2$ are homomorphic, and the image groups are soluble linear groups of degrees $r, s$ both $<n.$ By the induction hypothesis, there exist subgroups $H_1$ and $H_2$ of finite index in $G$ such that, after further transformation by a matrix of the from 
$$
\begin{bmatrix} t_1 & 0\\ 0 &t_2 \end{bmatrix} = t_3,
$$
every $(t_0 t_3)^{-1}y (t_0 t_3)$ with $y \in H_1 \cap H_2$ has the triangular form. But $|G: H_1 \cap H_2|$ is then finite and the result follows.

Next let $G$ be irreducible but suppose $G$ has a reducible normal subgroup $K$ such that not all the irreducible components of $K$ are equivalent. This means the following. Imagine $G$ to act as a group of linear transformations of the $n$-dimensional vector space $V$ over $\mathbb{K}$. Let $V_1$ be a subspace $\neq 0$ which is invariant under $K$ and is minimal with respect to this property, so that $K$ transforms $V_1$ irreducibly. If $x \in G$, then $V_1 x$ is also a minimal invariant subspace for $K$, since $K \nsub G$. If $\bar{V}$ is the subspace spanned by all minimal invariant subspaces of $K$, it follows that $\bar{V}$ is an invariant subspace of $G$ and so $\bar{V} = V.$ Therefore $V$ is the direct sum of a certain number of subspaces $V_1, V_2, \ldots, V_r$ each of which is a minimal invariant subspace for $K$. If in $V_i$ we choose any basis $b_1, \ldots, b_s$ and define, for $y \in K$, the matrix $\rho_i(y) = \eta$ by the equations 
$$
\eta \cdot b_j = \sum_{\ell = 1}^s n_j \ell b_\ell,
$$ then $\rho_i$ is an irreducible representation of $K$. If among the $r$ representations $\rho_1, \ldots, \rho_r$ of $K$, the first $t$ are equivalent but the rest are all inequivalent to $\rho_1$, then it is easy to show that every minimal invariant subspace $U$ of $K$ which gives rise to an irreducible representation equivalent to $\rho_1$ must be contained in $V_1 + V_2 + \ldots = W_1.$ And hence that if we group the remaining $V_i$ into sums of $K$-equivalent spaces in the same way, $W_1, W_2, \ldots, W_m$, then $V$ is the direct sum of $W_1, W_2, \ldots, W_m$ and each $x \in G$ permutes $W_1, \ldots, W_m$ among themselves. Accordingly, if $m > 1$, $G$ has a subgroup $G_1$ of finite index which leaves each $W_i$ invariant. $G_1$ is therefore reducible, and by the preceding case, $G_1$ has a subgroup $H$ of finite index which can be transformed to triangular form, and the result follows.

We may therefore assume that for every normal subgroup $K$ irreducible components of $K$ are all equivalent. In particular, if $K$ is Abelian, these irreducible components are all of degree $1$ since the field $\mathbb{K}$ is algebraically closed. Therefore every normal Abelian subgroup $A$ of $G$ is contained in $\mathbb{K} \cdot \text{Id}_{n \times n}$.

We may clearly suppose that $\mathbb{K} \cdot \text{Id}_{n \times n} \leq G,$ for if the group obtained by adjoining the scalar matrices to $G$ has the requisite property, $G$ will also have it. Also, since $\mathbb{K}$ is algebraically closed, if $G \geq \mathbb{K} \cdot \text{Id}_{n \times n}$, then $G = (\mathbb{K} \cdot \text{Id}_{n \times n}) \cdot G_1$ where $G_1$ is the subgroup of all matrices of determinant $1$ in $G$. $G_1$ is soluble if $G$ is, and the result holds for $G$ if it holds for $G_1$. Thus we may assume that all matrices of $G$ have determinant $1$.

Now let $N$ be any normal nilpotent subgroup $G$. Then $\zeta_1(N) \nsub G$ and $\zeta_1(N)$ is Abelian. Hence $\zeta_1(N) \leq \mathbb{K}\cdot \text{Id}_{n \times n}$ and so $N$ is finite by (\ref{9.3}). Let $$C = G^{(0} > G^{(1)} > \ldots > G^{(r)} = 1$$ be the derive of series of $G.$ $G^{(r-1)}$ is Abelian and normal in $G$, hence finite. Suppose for some $s>1$ that we have proved $G^{(s)}$ finite. Let $N$ be the centraliser of $G^{(s)}$ in $G^{(s-1)}$. Then $|G^{(s-1)} : N|$ is finite, since $G^{(s)}$ is finite. Also $N' \leq G^{(s)}$ since $N \leq G^{(s-1)}.$ Thus $[N',N] =1$ and $N$ is nilpotent of class $\leq 2.$ Further $N \nsub G$. Hence $N$ is finite. Therefore $G^{(s-1)}$ is finite. And so by induction, $G = G^{(0)}$ is finite and we may proved in all cases.
\end{proof} 

\begin{defn*}
Let $\mathcal{R}$ be the property of groups defined as follows: a group $G$ has $\mathcal{R}$ if and only if it has a subgroup $H$ such that $|G:H|$ is finite and $H'$ is nilpotent.
\end{defn*}

\begin{thm}[Malcev]\label{9.5}
(i) Every soluble linear group and (ii) every polycyclic group has the property $\mathcal{R}.$
\end{thm}
\begin{proof}
(i) is an immediate corollary of (\ref{9.4}). Let $G$ be a soluble linear group of degree $n$ over the field $\mathbb{K}$. We may assume $\mathbb{K}$ algebraically closed. By (\ref{9.4}), $G$ can be transformed so that some subgroup $H$ of finite index becomes a group $t^{-1}Ht = H_1$ consisting exclusively of triangular matrices. But then $H_1' \leq T_n(\mathbb{K})$ and $T_n(\mathbb{K})$ is nilpotent of class $<n.$ Hence $H'$ is also nilpotent of class $<n.$
    
(ii) Let $G$ be polycyclic. Then its derived series terminates with some $G^{(r)} = 1$ and every subgroup of $G$ is finitely generated. Since a finitely generated Abelian group $A$ has its periodic subgroup $B$ finite and $A/B$ is free Abelian of finite rank, $G$ will have a series of characteristic subgroups $$G = G_0 > G_1 > \ldots > G_s = 1$$ such that each $G_{i-1} / G_i$ is free Abelian of finite rank $r_i$ or else finite. In the former case, each element $x$ of $G$ induces an automorphism $\rho_i(x)$ of $G_{i-1} / G_i$ and the $\rho_i(x)$ may be regarded as unimodular matrices of degree $r_i.$ By extending from the field of rationals to an algebraically closed field $\mathbb{K}$ and making suitable transformation by a matrix $t_i$, (\ref{9.4}) yields a subgroup $H_i$ such that $|G:H_i|$ is finite and each $t_i^{-1} \rho_i(x) t$ with $x$ in $H_i$ is triangular. Hence $$
[G_{i-1}, H_i', H_i', \ldots, H_i']\leq G_i$$ if we use at least $r_i$ terms $H_i'$. In case $G_{i-1} / G_i$ is finite, define $H_i$ to be the centraliser of $G_{i-1} / G_i$ in $G$ and $r_i = 1.$  Again $|G: H_i|$ is finite. Let 
$$
H = \bigcap_{i=1}^sH_i.
$$
Then $|G:H|$ is finite; and for all $i$ we have 
$$
[G_{i+1},H',H', \ldots, H']=1
$$ if we use $r_1 + r_2 + \ldots r_s = t$ terms $H'$ and in particular $\gamma_{t+1}(H) = 1.$ Thus $H'$ is nilpotent.
\end{proof}

Let $U_n$ be the unimodular group of degree $n$. Malcev has also proved the following important results.
\begin{lemma}\label{9.6}
Every Abelian subgroup $A$ of $U_n$ is finitely generated.
\end{lemma}
\begin{proof}
Proof by induction on $n$. If $A$ is rationally reducible, its matrices may be transformed simultaneously to the form 
$$
y = \begin{bmatrix} y_1 & * \\ 0 & y_2 \end{bmatrix}
$$
with $y_1 \in U_r,$ $y_2 \in U_s$, $r+s= n,$ $0 < r < n,$ where $*$ stands for an $r \times s$ matrix of integers. Those $y$ for which both $y_1$ and $y_2$ are the appropriate unit matrix form a normal subgroup $A_0$ of $A$ which effectively is a subgroup of the additive group of all integral $r \times s$ matrices. The latter is Abelian of rank $rs$ and so $A_0$ is finitely generated. By the induction hypothesis, $A / A_0$ is also finitely generated. Hence so is $A$.

Suppose then that $A$ is rationally irreducible. Then by Schur's Lemma, the rational $n \times n$ matrices $x$ which commute with all $y \in A$ form a division algebra $D$ of finite degree over the rational field $\mathbb{K}_1$ and $A$ is contained in the centre $\mathbb{K}*$ of $D$. $\mathbb{K}*$ is therefore an algebraic number field and, since each $y \in A$ satisfies its characteristic equation and has integer coefficients and determinant $\pm 1,$ it follows that $A$ is contained in the group of units of $\mathbb{K}*$. By a classical theorem of Dirichlet, this group of units, and therefore $A$ also, is finitely generated.
\end{proof}

\begin{cor*} let $A$ be a finitely generated Abelian group. Then every Abelian subgroup of the group of automorphisms of $A$ is also finitely generated.
\end{cor*}
\begin{proof}
For $A$ is the extension of its periodic subgroup $B$ (which is finite) by $A/B$ which is free Abelian of some finite rank $r$ and the group of automorphisms of $A/B$ in $U_r.$ Hence if $G$ is an Abelian group of automorphisms of $A$, and $G_1$ the subgroup of all $a\in G$ such that $[A,a] \leq B,$ then $G/G_1$ is finitely generated by (\ref{9.6}). If $G_2$ consists of all $\beta \in G$ such that $[B, \beta] =1,$ then $|G:G_2|$ is finite. Let $G_0 = G_1 \cap G_2.$ Let $A = \left<B, a_1, a_2, \ldots, a_r\right>.$ Each $\gamma \in G_0$ determines the $r$ elements $b_i = [a_i, \gamma]$ of $B$ and since $[B, \gamma] =1,$ these $r$ elements determine $\gamma$ uniquely. Thus $G_0$ is finite. Since $G/G_0$ is finitely generated, it follows that $G$ is too.
\end{proof}

\begin{lemma}\label{9.7}
A soluble group $G$ is polycyclic if and only if all its Abelian subgroups are finitely generated.
\end{lemma}
\begin{proof}
Let $G$ be a soluble group with derived series of length $r$ such that every Abelian subgroup $A$ of $G$ is finitely generated. If $r=1$, $G$ is Abelian, hence finitely generated and therefore polycyclic. Suppose $r>1;$ take $A = G^{(r-1)}$ which is Abelian and therefore finitely generated. Let $B/A$ be any Abelian subgroup of $G/A$, and let $C$ be a maximal Abelian subgroup of $G$ which contains $A$. Such a group $C$ exists by Zorn's lemma and $C$ is finitely generated. Since $B/A$ is Abelian and $C \geq A$, $C$ is normal in $B$. By the maximal property of $C$, the centraliser of $C$ in $B$ is $C$ itself. Hence $B/C$ is isomorphic with an Abelian subgroup of the group of automorphisms of the finitely generated Abelian group $C$. By (\ref{9.6}) corollary, $B/C$ is finitely generated. Since $C$ is also finitely generated, so is $B$. Thus $G/A = G/G^{(r-1)}$ is a soluble group with derived series of length $r-1$ whose Abelian subgroups are all finitely generated. By the induction hypothesis, $G/G^{(r-1)}$ is polycyclic. Since $G^{(r-1)}$ is finitely generated and Abelian, $G$ itself is polycyclic. The converse is clear.
\end{proof}

\begin{thm}\label{9.8}
Every soluble subgroup of a unimodular group $U_n$ is polycyclic. \footnote{ This is the converse of L. Auslanders theorem (cf. p. 57) that every polycyclic group is embeddable in some $U_d$.}
\end{thm}
\begin{proof}
This now follows from (\ref{9.6}) and (\ref{9.7}). 
\end{proof}
More generally, it is easy to show that every soluble group of automorphisms of a polycyclic group is itself polycyclic. \footnote{ L. Auslander has proved that the full automorphism group of a polycyclic group is finitely related.}


\bibliography{bibs}

\begin{thebibliography}{1}

\bibitem{baumslag}
Gilbert Baumslag.
\newblock Lecture notes on nilpotent groups.
\newblock {\em (No Title)}, 1971.

\bibitem{hall_queen_mary}
P~Hall.
\newblock The edmonton notes on nilpotent groups, queen mary college math.
\newblock {\em Notes}, pages 415--462, 1969.

\bibitem{hartley}
B~Hartley.
\newblock Collected works of philip hall, 1993.

\bibitem{segal}
Daniel Segal.
\newblock {\em Polycyclic groups}.
\newblock Number~82. Cambridge University Press, 2005.

\end{thebibliography}
\bibliographystyle{plain}

\end{document}